\makeatletter
\newcommand{\bitvalue}[3]{%
  \newcount\bitvalue@temp
  \newcount\bitvalue@pos
  \newcount\bitvalue@mask
  \bitvalue@temp=#2\relax       
  \bitvalue@pos=#3\relax        
  \bitvalue@mask=1\relax
  \loop
    \ifnum\bitvalue@pos>0
      \multiply\bitvalue@mask by 2
      \advance\bitvalue@pos by -1
  \repeat
  \divide\bitvalue@temp by \bitvalue@mask
  \ifodd\bitvalue@temp
    \edef#1{1}%
  \else
    \edef#1{0}%
  \fi
}
\makeatother

\newcommand{\patternthree}[1]{%
  \bitvalue{\bA}{#1}{0}
  \bitvalue{\bB}{#1}{1}
  \bitvalue{\bC}{#1}{2}
  \psmatrix[colsep=1.5cm,rowsep=1.5cm,mnode=circle]
    a & b & c
    \everypsbox{\scriptstyle} 
    \ncarc[arcangle=30]{1,1}{1,2}_{\bC} 
    \ncarc[arcangle=30]{1,2}{1,3}_{\bA}
    \ncarc[arcangle=50]{1,1}{1,3}_{\bB} 
    \endpsmatrix
}

\newcommand{\patternfour}[1]{%
  \bitvalue{\bA}{#1}{0}
  \bitvalue{\bB}{#1}{1}
  \bitvalue{\bC}{#1}{2}
  \bitvalue{\bD}{#1}{3}
  \bitvalue{\bE}{#1}{4}
  \bitvalue{\bF}{#1}{5}
  \psmatrix[colsep=1.5cm,rowsep=1.5cm,mnode=circle]
    a & b & c & d
    \everypsbox{\scriptstyle}
    \ncarc[arcangle=30]{1,1}{1,2}_{\bF}
    \ncarc[arcangle=30]{1,2}{1,3}_{\bC}
    \ncarc[arcangle=30]{1,3}{1,4}_{\bA}
    \ncarc[arcangle=50]{1,1}{1,3}_{\bE}
    \ncarc[arcangle=50]{1,2}{1,4}_{\bB}
    \ncarc[arcangle=60]{1,1}{1,4}_{\bD}
  \endpsmatrix
}

\documentclass{article}
\usepackage{graphicx} 
\usepackage{amsthm}
\usepackage{amsfonts, amsmath, amssymb, amsthm}
\usepackage[hidelinks]{hyperref}
\usepackage{cleveref}
\usepackage{pst-all}
\usepackage{tikz}
\usetikzlibrary{decorations.markings}
\usepackage{pstricks,pst-node}
\usepackage{fp}
\usepackage{xparse}

\title{Ramsey-like theorems for separable permutations}
\author{Quentin Le Houérou \and Ludovic Patey}
\date{\today}

\newtheorem{theorem}{Theorem}
\numberwithin{theorem}{section}
\newtheorem{lemma}[theorem]{Lemma}
\newtheorem{proposition}[theorem]{Proposition}
\newtheorem{remark}[theorem]{Remark}
\newtheorem{definition}[theorem]{Definition}
\newtheorem{statement}[theorem]{Statement}
\newtheorem{corollary}[theorem]{Corollary}
\newtheorem{question}[theorem]{Question}
\newtheorem{example}[theorem]{Example}
\newtheorem{maintheorem}[theorem]{Main Theorem}

\makeatletter
\newtheorem*{rep@theorem}{\rep@title}
\newcommand{\newreptheorem}[2]{%
\newenvironment{rep#1}[1]{%
 \def\rep@title{#2 \ref{##1}}%
 \begin{rep@theorem}}%
 {\end{rep@theorem}}}
\makeatother

\newreptheorem{maintheorem}{Main Theorem}

\usepackage{xcolor}

\def\LL{\mathbb{L}}

\newcommand{\Psf}{\mathsf{P}}

\def\F{\mathcal{F}}
\def\U{\mathcal{U}}

\def\I{\mathcal{I}}

\def\L{\mathcal{L}}
\def\V{\mathcal{V}}
\def\A{\mathcal{A}}

\def\P{\mathcal{P}}
\def\M{\mathcal{M}}

\def\Q{\mathcal{Q}}

\newcommand{\dbf}{\mathbf{d}}

\newcommand{\R}{\mathcal{R}}

\newcommand{\cs}{2^{\NN}}

\newcommand{\bstr}{2^{<\NN}}

\newcommand{\uh}{\upharpoonright}

\newcommand{\ISig}{\mathsf{I}\Sigma^0}

\newcommand{\BSig}{\mathsf{B}\Sigma^0}

\newcommand{\RCA}[0]{\mathsf{RCA}}

\newcommand{\ADS}[0]{\mathsf{ADS}}
\newcommand{\EM}[0]{\mathsf{EM}}

\newcommand{\RT}[0]{\mathsf{RT}}

\newcommand{\COH}[0]{\mathsf{COH}}

\newcommand{\SADS}{\mathsf{SADS}}

\newcommand{\GP}{\mathsf{GP}}

\newcommand{\WWKLs}[1]{#1\mbox{-}\mathsf{WWKL}}
\newcommand{\DNCs}[1]{#1\mbox{-}\mathsf{DNC}}
\newcommand{\RANs}[1]{#1\mbox{-}\mathsf{RAN}}
\newcommand{\coRT}[0]{\mathsf{coRT}}
\newcommand{\SepRT}[0]{\mathsf{Sep}\mbox{-}\mathsf{RT}}

\newcommand{\NN}[0]{\mathbb{N}}

\newcommand{\card}{\operatorname{card}}

\def\qt#1{``#1''}%

\tikzset{negated/.style={
        decoration={markings,
            mark= at position 0.5 with {
                \draw[-, line width=0.8pt] (-2pt,2.5pt) -- (2pt,-2.5pt);
            }
        },
        postaction={decorate}
    }
}
\tikzset{maybe/.style={
        decoration={markings,
            mark= at position 0.5 with {
                \node[transform shape] (tempnode) {?};
            }
        },
        postaction={decorate}
    }
}

\def\qt#1{``#1''}%

\begin{document}

\maketitle

\begin{abstract}
We conduct a computability-theoretic study of Ramsey-like theorems of the form \qt{Every coloring of the edges of an infinite clique admits an infinite sub-clique avoiding some pattern}, with a particular focus on transitive patterns. As it turns out, the patterns corresponding to separable permutations play an important role in the computational features of the statement. We prove that the avoidance of any separable permutation is equivalent to the existence of an infinite homogeneous set in standard models, while this property fails for any other pattern. For this, we develop a novel argument for relativized diagonal non-computation.
\end{abstract}

\section{Introduction}

Ramsey's theorem for pairs and $k$ colors ($\RT^2_k$) is a foundational theorem from combinatorics, stating in its graph-theoretic form that every edge-coloring of an infinite clique using $k$ colors admits an infinite monochromatic sub-clique. From a meta-mathematical viewpoint, Ramsey's theorem for pairs stands out in reverse mathematics, as it constitutes the first counter-example of an important structural phenomenon known as the \qt{Big Five}~\cite{simpson2009subsystems,hirschfeldt2015slicing}.

Many real-world applications of Ramsey's theorem in mathematics involve degenerate colorings, such as transitive and semi-transitive colorings in automata theory~\cite{murakami2014ramseyan} and order theory~\cite{hirschfeldt2007combinatorial}, colorings with bounded homogeneous sets in graph theory~\cite{dilworth1950decomposition}. Other applications only require the existence of a weaker notion of homogeneity, such as transitive sets in~\cite{erdhos_representation_1964} in graph theory. Their meta-mathematical study in the context of computability-theory and reverse mathematics~\cite{hirschfeldt2007combinatorial,houerou2025reverse,lerman2013separating} revealed many similarities, suggesting the existence of general theorems for a larger family of statements.

Recently, Patey~\cite{patey2022ramseylike} started a program whose goal is to identify the structure of Ramsey-like theorems, partially ordered by the implication relation in reverse mathematics. He defined a syntactical family of Ramsey-like theorems, parameterized by forbidden patterns, and characterized those implying the Arithmetic Comprehension Axiom ($\mathsf{ACA}_0$) in standard models of reverse mathematics. This characterization proved to be useful in the study of further statements, such as~\cite[Chapter 4]{solda2021calibrating}. A full completion of this program seems currently out of reach, but some non-trivial characterizations were obtained by Mimouni and Patey~\cite{mimouni2025ramseylike}.

In this article, we pursue the computability-theoretic study of weakenings of Ramsey's theorem for pairs by considering infinite sub-cliques avoiding some finite coloring patterns instead of monochromatic sub-cliques. For instance, the weakening of $\RT^2_2$ to cliques avoiding the following two patterns yields the Erd\H{o}s-Moser theorem ($\EM$), a statement from graph theory on infinite tournaments.
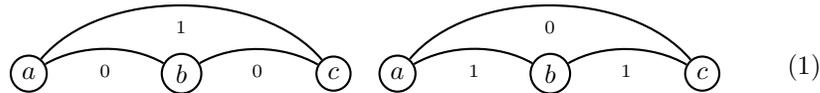
\vspace{1cm}
\begin{equation}\label{eq:non-transitivity}
\patternthree{2} \hspace{10pt} \patternthree{5}
\end{equation}
We also consider dual weakenings of Ramsey's theorem for pairs by restricting the colorings to those avoiding some pattern. For instance, $\RT^2_2$ for colorings avoiding (\ref{eq:non-transitivity}) yields the Ascending Descending Sequence theorem ($\ADS$), stating that every infinite linear order admits an infinite ascending or descending sequence. The statements $\RT^2_2$, $\EM$ and $\ADS$ have been extensively studied in reverse mathematics~\cite{seetapun1995strength,cholak2001strength,hirschfeldt2007combinatorial,lerman2013separating}. More recently, there has been a more systematic study of Ramsey-like theorems for arbitrary patterns~\cite{patey2022ramseylike,cervelle2024reverse,mimouni2025ramseylike}.

Our goal is to characterize the patterns whose avoidance implies Ramsey's theorem for pairs. The Erd\H{o}s-Moser theorem states the existence of an infinite set avoiding all the non-transitive patterns, and is known to be strictly weaker than Ramsey's theorem for pairs~\cite{lerman2013separating}. We shall therefore give a particular focus on transitive forbidden patterns, that is, patterns which do not include (\ref{eq:non-transitivity}). These patterns are in one-to-one correspondence with permutations. Among these, the class of \emph{separable permutations} introduced by Bose, Buss and Lubiw~\cite{bose1998pattern} plays an important role in studying the computability-theoretic features of Ramsey-like statements. Our main result is a characterization of the patterns whose avoidance implies Ramsey's theorem for pairs, as the separable permutations (\Cref{maintheorem1}). For this, we use an elaborate computability-theoretic analysis of a probabilistic construction.


\subsection{Reverse mathematics}

Reverse mathematics is a meta-mathematical program founded by Harvey Friedman, whose goal is to find optimal axioms to prove ordinary theorems. It uses the framework of subsystems of second-order arithmetic, with a base theory, $\RCA_0$, capturing \emph{computable mathematics}. More precisely, the \emph{$\Delta^0_1$-comprehension scheme} is the following scheme, for every $\Sigma^0_1$-formula $\varphi(x)$ and every $\Pi^0_1$-formula $\psi(x)$:
$$
\forall x(\varphi(x) \leftrightarrow \psi(x)) \to \exists A \forall x(x \in A \leftrightarrow \varphi(x))
$$
The \emph{$\Sigma^0_1$-induction scheme} is the following scheme, for every $\Sigma^0_1$-formula $\varphi(x)$:
$$
\varphi(0) \wedge \forall x(\varphi(x) \to \varphi(x+1)) \to \forall y \varphi(y)
$$
$\RCA_0$ consists of Robinson arithmetic (in other words, Peano arithmetic without induction), together with the $\Delta^0_1$-comprehension scheme and the $\Sigma^0_1$-induction scheme. See Simpson~\cite{simpson2009subsystems} for an extensive study of early reverse mathematics, and Dzhafarov and Mummert~\cite{dzhafarov2022reverse} for a modern presentation with a particular focus on combinatorics.

A structure in the language of second-order arithmetic is of the form 
$$(M, S, 0, 1, +, \cdot, <)$$
where $M$ denotes the set of integers of the model, $S \subseteq \P(M)$ is a collection of sets of integers representing the second-order part, $0, 1 \in M$ are distinguished constants, and $+$, $\cdot$ and $<$ represent the addition, multiplication and order, respectively. We are particularly interested in \emph{$\omega$-structures}, that is, structures where $M = \omega$ are the standard integers, and $0, 1, +, \cdot, <$ have the usual interpretation. An $\omega$-structure is therefore fully specified by its second-order part~$S$.

$\omega$-models of $\RCA_0$ admit a nice characterization in terms of Turing ideals. A \emph{Turing ideal} $\I$ is a non-empty class of sets which is closed downward under the Turing reduction (if $X \in \I$ and $Y \leq_T X$, then $Y \in \I$), and closed under the effective join (if $X, Y \in \I$, then $X \oplus Y = \{ 2n : n \in X \} \cup \{ 2n+1 : n \in Y \} \in \I$). An $\omega$-structure is a model of~$\RCA_0$ if and only if its second-order part is a Turing ideal. We shall exclusively consider $\omega$-structures which are models of $\RCA_0$, and simply call them \emph{$\omega$-models}.

\subsection{$\RT^2_2$-patterns}

Given a finite or infinite set~$X \subseteq \NN$, we write $[X]^2$ for the set of all unordered pairs over~$X$. We identify any integer~$\ell$ with the set~$\{0, \dots, \ell-1\}$ and write $[\ell]^2$ for the set of all unordered pairs of elements smaller than~$\ell$. We also identify the set~$[X]^2$ with the set of all increasing ordered pairs over~$X$. Therefore, given a function $f : [X]^2 \to k$, we write $f(x, y)$ for $f(\{x, y\})$, assuming $x < y$. A set $H \subseteq \NN$ is \emph{$f$-homogeneous} (for color~$i < k$) if for every~$\{x,y\} \in [H]^2$, $f(x, y) = i$.

\begin{statement}[Ramsey's theorem for pairs]
$\RT^2_k$ is the statement \qt{For every coloring $f : [\NN]^2 \to k$, there exists an infinite homogeneous set.}
\end{statement}

An \emph{$\RT^2_2$-pattern} or simply \emph{pattern} (of size~$\ell$) is a function $p : [\ell]^2 \to 2$. We then write $|p|$ for its size~$\ell$.
Fix a function $f : [\NN]^2 \to 2$. We say that a finite set $F = \{ x_0 < \dots < x_{\ell-1} \}$ \emph{$f$-realizes} $p$ if for every~$i < j$, $f(x_i, x_j) = p(i,j)$. We say that~$H \subseteq \NN$ \emph{$f$-avoids} the pattern $p$ if no subset of~$H$ $f$-realizes~$p$. When~$H = \NN$, we simply say that the coloring $f$ \emph{avoids~$p$}. We are primarily interested in the following two families of statement:

\begin{statement}
Given a pattern~$p$, $\RT^2_2(p)$ is the statement \qt{For every coloring $f : [\NN]^2 \to 2$, there exists an infinite set $f$-avoiding~$p$.}
\end{statement}

\begin{statement}
Given a pattern~$p$, $\coRT^2_2(p)$ is the statement \qt{For every coloring $f : [\NN]^2 \to 2$ avoiding~$p$, there exists an infinite homogeneous set.}
\end{statement}

Note that if the pattern~$p$ is a constant function~$i < 2$, then the statement $\RT^2_2(p)$ is combinatorially false, as witnessed by the function $f : [\NN]^2 \to 2$ also equal to~$i$ on its whole domain.
By Ramsey's theorem for pairs, $\RT^2_2(p)$ is true for any other pattern~$p$.
The statements $\RT^2_2(p)$ and $\coRT^2_2(p)$ play a dual role with respect to~$\RT^2_2$, as $\RCA_0 \vdash \RT^2_2 \leftrightarrow \RT^2_2(p) + \coRT^2_2(p)$ for any non-constant pattern~$p$. At first sight, the following Erd\H{o}s-Moser theorem cannot be cast as one of the previous statements, as it requires to avoid multiple patterns simultaneously.

\begin{statement}[Erd\H{o}s-Moser theorem]
$\EM$ is the statement \qt{For every coloring $f : [\NN]^2 \to 2$, there exists an infinite set $f$-avoiding both non-transitivity patterns (see (\ref{eq:non-transitivity}))}. 
\end{statement}

Every pattern~$p : [\ell]^2 \to 2$ admits a \emph{dual pattern} $\overline{p} : [\ell]^2 \to 2$ defined by $\bar p(x, y) = 1-p(x,y)$. In particular, the non-transitivity patterns (\ref{eq:non-transitivity}) are dual. Over $\RCA_0$, $\RT^2_2(p)$ and $\RT^2_2(\bar p)$ are equivalent, as given a coloring $f : [\NN]^2 \to 2$, one can always consider the dual coloring $\bar f : [\NN]^2 \to 2$ and apply $\RT^2_2(\bar p)$. This enables to define the Erd\H{o}s-Moser theorem as a Ramsey-like statement: Over $\RCA_0$, $\EM$ is equivalent to the statement $\RT^2_2(p)$, where $p$ is any of~(\ref{eq:non-transitivity}).

\subsection{Permutations}

A \emph{permutation} on~$\ell$ is a bijection $\pi : \ell \to \ell$. A permutation~$\pi$ can be seen as a linear order $\L_\pi = (\ell, <_\pi)$ defined for every~$\{x, y\} \in [\ell]^2$ by $x <_\pi y$ iff $\pi(x) <_\NN \pi(y)$.
There is therefore a one-to-one correspondence between permutations on~$\ell$ and patterns of size~$\ell$ avoiding both non-transitivity patterns (see (\ref{eq:non-transitivity})).

Indeed, given a permutation $\pi$ on~$\ell$, one can define a unique pattern~$p_\pi : [\ell]^2 \to 2$
coding the graph of~$<_\pi$, that is, for $x <_\NN y <_\NN \ell$, $p_\pi(x, y) = 0$ if $x <_\pi y$ and $p_\pi(x, y) = 1$ otherwise. Thanks to this correspondence, we shall consider permutations as particular $\RT^2_2$-patterns, and will use either representation when most convenient. Permutations enjoy a simple notation as their range sequence. For instance, the sequence $2031$ denotes the permutation~$\pi$ on~$4$ defined by $\pi(0) = 2$, $\pi(1) = 0$, $\pi(2) = 3$ and $\pi(3) = 1$. Under the pattern representation, this corresponds to
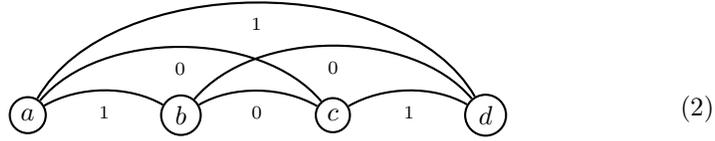
\vspace{1.3cm}
\begin{equation}\label{eq:pattern-2031}
\patternfour{41}    
\end{equation}

Permutations enjoy two natural associative operators, namely, the skew sum $\ominus$ and the direct sum $\oplus$, combining shorter permutations into longer ones. They are defined as follows: for $\pi$ and $\sigma$ two permutations of size $m$ and $n$, respectively,

$$(\pi\ominus\sigma)(x) = \begin{cases} \pi(x)+n & \text{for }x < m, \\
\sigma(x-m) & \text{for }m\le x < m+n,\end{cases}
$$

$$(\pi\oplus\sigma)(x) = \begin{cases} \pi(x) & \text{for }x < m,\\
\sigma(x-m)+m & \text{for }m \le x < m+n.\end{cases}
$$

These sums enable to define the class of separable permutations, which will play a central role in this article.

\begin{definition}
    A permutation is \emph{separable} if it can be obtained from the trivial permutation 0 by direct sums and skew sums.
\end{definition}

Bose, Buss and Lubiw~\cite{bose1998pattern} studied separable permutations, and gave a very nice characterization in terms of pattern avoidance, that we shall study in~\Cref{sec:irreducible-convergent}.
As it turns out, many lower bounds on the strength of $\RT^2_2(p)$ for separable permutations are uniform in~$p$, motivating the following definition.

\begin{statement}
$\SepRT^2_2$ is the statement \qt{For every coloring $f : [\NN]^2 \to 2$, there exists an infinite set $f$-avoiding some separable pattern.}
\end{statement}

In particular, for every separable permutation~$p$, $\RCA_0 \vdash \RT^2_2(p) \to \SepRT^2_2$. Moreover, 
$\RCA_0 \vdash (\SepRT^2_2 \wedge \forall p \mbox{ separable permutation } \coRT^2_2(p)) \to \RT^2_2$.
Note that in non-standard models, separable permutations might be of non-standard size, so the quantification is internal. 

We prove that $\SepRT^2_2$ implies $\RT^2_2$ over $\omega$-models, and combine this result with a theorem from Mimouni and Patey~\cite{mimouni2025ramseylike} to deduce the following characterization.

\begin{maintheorem}\label{maintheorem1}
For every pattern~$p$, $\RT^2_2(p)$ implies $\RT^2_2$ over $\omega$-models if and only if $p$ is a separable permutation.
\end{maintheorem}

\subsection{Organization}

\Cref{maintheorem1} is a characterization for which both directions are highly non-trivial. However, in one direction, most of the work is already present in Patey and Mimouni~\cite{mimouni2025ramseylike}. Indeed, they used effective forcing to prove that Ramsey-like theorems for a syntactical family of patterns are strictly weaker than Ramsey's theorem for pairs in $\omega$-models. In \Cref{sec:irreducible-convergent}, we prove that this family of patterns corresponds exactly to the patterns which are either non-transitive, or do not induce separable permutations. We state multiple combinatorial characterizations of the class of separable permutations, and derive some computability-theoretic consequences which will be useful for the remainder the article.

The proof of the other direction is divided into multiple sections, as it involves many concepts and requires multiple tools from algorithmic randomness. The goal is to show that $\SepRT^2_2$ implies Ramsey's theorem for pairs in $\omega$-models. As a warm-up, we give in \Cref{sec:minimal-covers} a simple and direct proof of a weaker result, that is, $\SepRT^2_2$ implies $\ADS$ over $\omega$-models, and deduce that $\RT^2_2$ forms a minimal cover of~$\EM$ in the partial order of Ramsey-like statements.

Patey and Mimouni~\cite{mimouni2025ramseylike} proved that every Ramsey-like theorem implies the existence of a $\emptyset'$-DNC degree (defined in \Cref{sec:2dnc-sep}), so it suffices to prove that if $p$ is a separable permutation, then every $\emptyset'$-DNC degree computes a solution to any computable instance of $\coRT^2_2(p)$. To do so, we define in \Cref{sec:basis-sep-permutations} the notion of $k$-fractal, which is a family of canonical separable permutations enjoying some basis properties. This notion is used in \Cref{sec:sep-perm-measure} to prove that $\coRT^2_2(p)$ admits probabilistic solutions for every separable permutation~$p$. The argument is then refined in \Cref{sec:2dnc-sep} to prove that every $\emptyset'$-DNC degree computes a solution to~$\coRT^2_2(p)$, and deduce \Cref{maintheorem1}.

In \Cref{sec:linear-order-measure}, we study the role of randomness in computing infinite ascending or descending sequences for computable linear orders, and combine it with \Cref{sec:sep-perm-measure} to obtain new results. In \Cref{sec:sep-perm-first-order}, we study the first-order part of Ramsey-like theorems for separable permutations. Finally, we conclude with open questions in \Cref{sec:open-questions}.



\subsection{Notation}

Throughout this article, we use the standard notations from computability theory (see Soare~\cite{soare2016turing}) and reverse mathematics (see Dzhafarov and Mummert~\cite{dzhafarov2022reverse}).
In particular, $\Phi_0, \Phi_1, \dots$ is an enumeration of all Turing functionals.

\emph{Sets}. We write in upper case latin letters $X, Y, Z, \dots$ sets of integers. Usually, $E, F$ denote finite sets. We write $X < Y$ for $\forall x \in X\ \forall y \in Y\ x < y$.

\emph{Binary strings}. We let $2^{<\NN}$ denote the set of all finite binary strings, and $2^\NN$ denote the class of all infinite binary sequences. Note that $2^\NN$ is in bijection with $\P(\NN)$ and both are usually identified.
Finite binary strings are written with greek letters~$\sigma, \tau, \mu, \dots$. The length of a binary string $\sigma$ is written~$|\sigma|$, and the concatenation of two binary strings $\sigma$ and $\tau$ is written $\sigma \cdot \tau$. Given an infinite binary sequence~$X \in 2^\NN$ and $n \in \NN$, we write $X \uh_n$ for its initial segment of length~$n$. A string~$\sigma$ is a prefix of a string~$\tau$, written $\sigma \preceq \tau$, if there is some~$\mu$ such that $\sigma \cdot \mu = \tau$.

\emph{Coding}. We write $\langle \cdot, \cdot \rangle : \NN^2 \to \NN$ for the Cantor pairing function. This extends to arbitrary finite sequences by letting $\langle a_0, \dots, a_{k-1} \rangle = \langle a_0, \langle a_1, \dots, a_{k-1} \rangle$. Accordingly, $\langle\rangle$ denotes the empty sequence.

\section{Irreducible and convergent patterns}\label[section]{sec:irreducible-convergent}

Bose, Buss and Lubiw~\cite{bose1998pattern} introduced the notion of separable permutation as a natural family for which multiple problems (sub-pattern decision, longest common sub-pattern) admit a polynomial time procedure. 
The goal of this section is to state multiple (existing and new) combinatorial characterizations of the class of separable permutations. We then derive some immediate computability-theoretic consequences which will be useful for the remainder of this article. Separable permutations admit the following historical characterization:


\begin{theorem}[Bose, Buss and Lubiw~\cite{bose1998pattern}]\label[theorem]{thm:characterization-sep-bose}
A permutation is separable if and only if it does not contain the permutations $1302$ and $2031$ as sub-patterns.
\end{theorem}

The main interest of this characterization in this article is the existence of a finite number of patterns (the two non-transitivity patterns and the permutations $1302$ and $2031$) such that every coloring avoiding all of them also avoids all non-separable patterns.
Among the consequences, the following corollary states that over~$\RCA_0$, being a model of the Erd\H{o}s-Moser theorem and of $\RT^2_2(1302)$ is sufficient to satisfy $\RT^2_2(p)$ for any other $\RT^2_2$-pattern~$p$, except for the separable permutations. This is actually an equivalence as $\EM$ and $\RT^2_2(1302)$ are both equivalent over~$\RCA_0$ to a statement of the form $\RT^2_2(p)$ for such a pattern~$p$. \Cref{cor:non-sep-finite-basis} will be very useful for proving tightness of some theorems about $\SepRT^2_2$ in \Cref{sec:linear-order-measure}.

\begin{corollary}\label[corollary]{cor:non-sep-finite-basis}
For every $\RT^2_2$-pattern~$p$ which is not a separable permutation, $\RCA_0 + \EM + \RT^2_2(1302) \vdash \RT^2_2(p)$.
\end{corollary}
\begin{proof}
First, note that $\RCA_0 \vdash \RT^2_2(1302) \leftrightarrow \RT^2_2(2031)$ as these patterns are dual.
Fix an $\RT^2_2$-pattern~$p$ which is not a separable permutation.
Let $f : [\NN]^2 \to 2$ be a coloring. By $\EM$, $\RT^2_2(1302)$ and $\RT^2_2(2031)$, there is an infinite subset $H \subseteq \NN$ $f$-avoiding the non-transitivity $\RT^2_2$-patterns (see (\ref{eq:non-transitivity}))
and the permutations 1302 and 2031. It follows from \Cref{thm:characterization-sep-bose} that the only $\RT^2_2$-patterns which $f$-appear in~$H$ are separable permutations, hence $H$ is an $\RT^2_2(p)$-solution to~$f$.
\end{proof}

Mimouni and Patey~\cite{mimouni2025ramseylike} introduced two operators on $\RT^2_2$-patterns, playing an important role in the computability-theoretic study of these statements. The \emph{join} $p \uplus q$ of two patterns $p$ and $q$ is the pattern of size $|p|+|q|-1$ defined for every~$x < y < |p|+|q|-1$ by
$$
(p \uplus q)(x, y) = \left\{ \begin{array}{ll}
    p(x, y) & \mbox{ if } y < |p|\\
    q(x-|p|+1, y-|p|+1) & \mbox{ if } x \geq |p|-1\\
    p(x, |p|-1) & \mbox{ if } x < |p|-1 \mbox{ and } y > |p|-1
\end{array}\right.
$$
They proved over $\RCA_0$ that for every computable coloring $f : [\NN]^2 \to 2$, every infinite set~$H$ $f$-avoiding $p \uplus q$ computes an infinite subset~$G \subseteq H$ $f$-avoiding either~$p$ or $q$. It follows that $\RCA_0 \vdash \coRT^2_2(p) \wedge \coRT^2_2(q) \to \coRT^2_2(p \uplus q)$ for every pair of patterns~$p$ and~$q$.

\begin{definition}
A pattern $p$ is \emph{reducible} if it is of the form $p_0 \uplus p_1$ for two patterns $p_0$ and $p_1$ of size at least~2. Otherwise, $p$ is \emph{irreducible}.
\end{definition}

Given a pattern~$p : [\ell]^2 \to 2$ and a color~$c < 2$, we write $p \rhd c$ for the pattern of size~$\ell+1$ such that $(p \rhd c)(x, y) = p(x, y)$ if $x < y < \ell$ and $(p \rhd c)(x, \ell) = c$ if $x < \ell$.

\begin{definition}
A pattern $p$ is \emph{convergent} if it is of the form $p \rhd i$ for some pattern~$p$ and some $i < 2$. Otherwise, it is \emph{divergent}.
\end{definition}

The skew sum and direct sum can be reformulated in terms of the join and convergence operators.
Indeed, if $p$ and $q$ are two permutations, then $p \oplus q = (p \rhd 0) \uplus q$ and $p \ominus q = (p \rhd 1) \uplus q$. Mimouni and Patey~\cite{mimouni2025ramseylike} gave a purely computability-theoretic characterization of the patterns which admit a diverging and irreducible sub-pattern, in terms of preservation of $\omega$ hyperimmunities.
A function $f : \NN \to \NN$ is \emph{($X$-)hyperimmune} if it is not dominated by any ($X$-)computable function.

\begin{definition}
A problem~$\Psf$ admits \emph{preservation of $\omega$ hyperimmunities} if for every set~$Z$,
every countable sequence $f_0, f_1, \dots$ of $Z$-hyperimmune functions and every $Z$-computable $\Psf$-instance~$X$, there is a $\Psf$-solution~$Y$ to~$X$ such that $f_0, f_1, \dots$ are all $Y \oplus Z$-hyperimmune.
\end{definition}

This notion of preservation and some variants were introduced by Patey~\cite{patey2017iterative} who proved that the Erd\H{o}s-Moser theorem ($\EM$) admits preservation of $\omega$ hyperimmunities, while the Ascending Descending Sequence principle ($\ADS$) does not. More recently, Mimouni and Patey~\cite{mimouni2025ramseylike} characterized the Ramsey-like statements admitting preservation of $\omega$ hyperimmunities, in terms of irreducibility and divergence:

\begin{theorem}[Mimouni and Patey~\cite{mimouni2025ramseylike}]\label[theorem]{thm:preserve-hyps-irr-div}
Let $p$ be a pattern. $\RT^2_2(p)$ admits preservation of $\omega$ hyperimmunities if and only if $p$ contains a sub-pattern which is simultaneously diverging and irreducible.
\end{theorem}

We now prove that this class of patterns corresponds exactly to the separable permutations.

\begin{proposition}\label[proposition]{prop:separable-no-diverging-irreducible}
A permutation is separable if and only if it does not contain any sub-pattern which is simultaneously diverging and irreducible.
\end{proposition}

\begin{proof}
It is clear that the permutations $2031$ and $1302$ are diverging and irreducible.
Thus, every permutation without a diverging irreducible sub-pattern is separable.

For the other inclusion, it is sufficient to show that any diverging irreducible permutation pattern admits $2031$ or $1302$ as a sub-pattern. Thus, consider such a permutation $x_0, \dots, x_n$ and assume by symmetry that $x_{n-1} < x_n$ (the case $x_{n-1} > x_n$ yields the dual sub-pattern). Since this permutation is diverging, there exists $i$ such that $x_i > x_n$ and consider the biggest such $i$. Then, consider $j \in \{i+1, \dots, n-1\}$ such that $x_j$ is the smallest possible.

As the permutation is irreducible, it cannot be written as the join of $\{x_0, \dots,\allowbreak x_{i}, x_{i+1}\}$ and $\{x_{i+1}, \dots, x_n\}$, meaning that there exists some element $x_k$ for $k \leq i$ that is bigger than one element of $\{x_{i+1}, \dots, x_n\}$ and smaller than one other, thus $x_j < x_k < x_n$ (by definition of $x_i$, $x_n$ is the biggest element of that set). It is not possible to have $k = i$, as $x_i > x_n$, thus $k < i < j < n$ and the permutation type of $x_k, x_i, x_j, x_n$ is $1302$.
\end{proof}

\begin{corollary}\label[corollary]{cor:sep-permutation-hyperimmunities}
Let $p$ be a pattern. $\RT^2_2(p)$ admits preservation of $\omega$ hyperimmunities if and only if $p$ is not a separable permutation.
\end{corollary}
\begin{proof}
Immediate by \Cref{thm:preserve-hyps-irr-div} and \Cref{prop:separable-no-diverging-irreducible}. Note that if a pattern is not a permutation, then it contains one of two non-transitivity patterns (see (\ref{eq:non-transitivity})), which are both divergent and irreducible.
\end{proof}

The following corollary is a consequence of the fact that $\RT^2_2(p)$ preserves $\omega$ hyperimmunities for every separable permutation~$p$, while $\RT^2_2$ does not, and uses a general construction over preservation properties (see \cite[Section 3.4, Lemma 3.4.2]{patey2016reverse}) We however reprove it in the particular case of hyperimmunities, for the sake of completeness.

Hyperimmune degrees can also be defined in terms of sets. The canonical code of a finite set~$F$ is the integer $\sum_{n \in F} 2^n$.
A \emph{c.e.} array is a c.e.\ sequence of pairwise disjoint, canonically coded, non-empty finite sets. An infinite set~$A \subseteq \NN$ is \emph{($X$-)hyperimmune} if for every ($X$-)c.e.\ array $(F_n : n \in \NN)$, $F_n \cap A = \emptyset$ for some~$n \in \NN$.
Equivalently, an infinite set~$A$ is $X$-hyperimmune if its principal function is $X$-hyperimmune, where the principal function of~$A$ is the function $p_A : \NN \to \NN$ which to $n$ associates its $n$th element.

\begin{corollary}\label[corollary]{cor:sep-permutation-not-rt22}
Let $p$ be a pattern that is not a separable permutation. Then $\RT^2_2(p)$ does not imply $\RT^2_2$ over $\omega$-models.
\end{corollary}
\begin{proof}
First, notice that $\RT^2_2$ does not even preserve 2 hyperimmunities.
Indeed, let $A$ be a $\Delta^0_2$ set which is bi-hyperimmune (both $A$ and $\overline{A}$ are hyperimmune). By Shoenfield's limit lemma, there is a computable function $f : [\NN]^2 \to 2$ such that for every~$x \in \NN$, $\lim_y f(x, y) = A(x)$. Consider $f$ as a (stable) computable instance of~$\RT^2_2$. Note that every infinite $f$-homogeneous set~$H$ is a subset of~$A$ or of $\overline{A}$, in which case either~$A$ or $\overline{A}$ is not $H$-hyperimmune.

We now construct an $\omega$-model~$\M$ of $\RT^2_2(p)$ such that for every~$X \in \M$, $A$ and $\overline{A}$ are both $X$-hyperimmune. Since $f$ is computable, $f \in \M$, but $\M$ does not contain any infinite $f$-homogeneous set, so $\M \not \models \RT^2_2$.
More precisely, we construct an infinite sequence $Z_0 \leq_T Z_1 \leq_T \dots$ of sets such that for every~$n \in \NN$,
\begin{itemize}
    \item[(1)] If~$n = \langle a, b\rangle$ and $\Phi^{Z_a}_b$ is a coloring $g : [\NN]^2 \to 2$, then $Z_{n+1}$ computes an infinite set $g$-avoiding~$p$;
    \item[(2)] $A$ and $\overline{A}$ are both~$Z_n$-hyperimmune.
\end{itemize}
First, let $Z_0 = \emptyset$. Assuming $Z_n$ is defined, letting $n = \langle a, b \rangle$, if $\Phi^{Z_a}_b$ is not a coloring, then let $Z_{n+1} = Z_n$. Otherwise,  by \Cref{cor:sep-permutation-hyperimmunities}, there is an infinite set~$H$ $\Phi^{Z_a}_b$-avoiding~$p$ such that $A$ and $\overline{A}$ are both $H \oplus Z_n$-hyperimmune. Let $Z_{n+1} = H \oplus Z_n$. This completes our construction.
Let $\M$ be the $\omega$-model whose second-order part is $\{ X : \exists n\ X \leq_T Z_n \}$. By (1), $\M \models \RT^2_2(p)$, and by (2), $f$ has no homogeneous set in~$\M$, so $\M \not \models \RT^2_2$.
\end{proof}

We shall see in \Cref{sec:2dnc-sep} that the previous corollary is tight, in that if $p$ is a separable permutation, then $\RT^2_2(p)$ implies $\RT^2_2$ over $\omega$-models (\Cref{maintheorem1}).

\section{Separable permutations and minimal covers}\label[section]{sec:minimal-covers}

The statements $\RT^2_2(p)$ and $\coRT^2_2(p)$ were studied both classically and from  a reverse mathematical viewpoint for small patterns~$p$. For instance, $\coRT^2_2(120)$ corresponds the restriction of the Chain-Antichain principle to trees~\cite{cervelle2024reverse}. However, for large patterns, the situation is less clear, as the corresponding statements were not studied previously, and the induced combinatorial objects are therefore new and seem out of reach of intuition.

To guide their computability-theoretic study, it is therefore natural to consider the known properties common to all small patterns, and try to generalize them to all sizes. Let us consider separable permutations of small size. By \Cref{thm:characterization-sep-bose}, every permutation of size at most 3 is separable. The only permutations of size~2 are $01$ and $10$, and correspond to constant patterns. Recall that the statement $\RT^2_2(p)$ for a constant pattern is combinatorially false, so the statements $\RT^2_2(01)$ and $\RT(10)$ are of no interest. The Ramsey-like statements for permutations of size~3 were studied by multiple authors, and summarized in Cervelle, Gaudelier and Patey~\cite[Section 1.2]{cervelle2024reverse}.
There are 6 permutations of size~3, which can be grouped by duals: the permutations $012 = 01 \uplus 01$ and $210 = 10 \uplus 10$ which correspond to constant patterns, the permutations $021 = 01 \uplus 10$ and $201 = 10 \uplus 01$ which are reducible and diverging, and the permutations $120$ and $102$, which are irreducible and converging. 

By Mimouni and Patey~\cite{mimouni2025ramseylike}, the statement $\RT^2_2(p \uplus q)$ is equivalent to the existence, for every coloring $f : [\NN]^2 \to 2$, of an infinite set which $f$-avoids~$p$ or~$q$. Therefore, $\RT^2_2$ is computably reducible to $\RT^2_2(p)$ for every reducible permutation of size~3, and the main cases of interest are for the irreducible permutations $120$ and $102$.
Since $\RT^2_2(120 \uplus 102)$ is equivalent to the statement \qt{for every coloring $f : [\NN]^2 \to 2$, there is an infinite set $f$-avoiding either $120$ or $102$}, it follows that $\RCA_0 \vdash \RT^2_2(120 \uplus 102) \wedge \coRT^2_2(120) \wedge \coRT^2_2(102) \to \RT^2_2$.
Cervelle, Gaudelier and Patey~\cite{cervelle2024reverse} proved that $\RCA_0 \vdash \RT^2_2(120 \uplus 102) \to \ADS$, and that $\RCA_0 \vdash \ADS \to (\coRT^2_2(120) \wedge \coRT^2_2(102))$. Thus $\RCA_0 \vdash \RT^2_2(120 \uplus 102) \to \RT^2_2$. To summarize, all the statements $\RT^2_2(p)$ for separable permutations~$p$ of size at most~3 are known to imply~$\ADS$ over $\RCA_0$ with only one application.

We now prove that $\SepRT^2_2$ implies $\ADS$ over~$\omega$-models. In some sense, one can think of $\SepRT^2_2$ as $\RT^2_2(\biguplus P)$ where $P$ is the set of all separable permutations. This intuition has two limits: first, the operator $\biguplus$ yields infinite patterns when applied on infinite sets; second, in non-standard models, a solution to an instance of $\SepRT^2_2$ may avoid some separable patterns of non-standard size. However, for a finite set~$P$ of standard separable permutations, the permutation $\biguplus P$ is separable and of standard size, and the proof can be formalized over~$\RCA_0$ to obtain $\RCA_0 \vdash \RT^2_2(\biguplus P) \to \ADS$.

\begin{proposition}\label[proposition]{prop:separable-trivial-ads}
Let $\L = (\NN, <_\L)$ be a computable linear order avoiding a separable permutation~$p$.
Then there is an infinite computable $\L$-ascending or $\L$-descending sequence.
\end{proposition}
\begin{proof}
We show by induction on~$p$ that for every computable linear order $\L = (\NN, <_\L)$ without any infinite computable $\L$-increasing or $\L$-decreasing sequence, the pattern $p$ appears.
More precisely, we proceed by structural induction on the construction of $p$ by a series of direct and skew summation, starting from the trivial permutation with one element.
This statement holds for the trivial permutation. The direct sum and the skew sum playing a dual role, we prove the inductive step only for the direct sum.

Let $p$ be the direct sum of two permutations $p_0$ and $p_1$ satisfying the induction hypothesis and let $\L = (\NN,<_\L)$ be a computable linear order without any infinite computable $\L$-increasing or $\L$-decreasing sequence, hence both of these pattern appear in~$\L$.

If the pattern $p_0$ appears in~$\L$ at position $F_0$ with the element $x_0 = \max_\L F_0$ such that the set $\{x \in \NN : x >_\L x_0\}$ is infinite, then we are done, as we can consider the induced order $(\{x \in \NN : x >_\L x_0 \wedge x > F_0\}, <_\L)$ and apply the induction hypothesis on it (by first taking a computable bijection between this order and one having $\NN$ as base set), thus, the pattern $p_1$ appears in that order at some position $F_1$, and by construction, $F_0 < F_1$ and $F_0 <_\L F_1$ holds, thus $F_0 \cup F_1$ $\L$-realizes $p_0 \oplus p_1$.

If for every set $F_0$ $\L$-realizing $p_0$, we have that $x_0 = \max_\L F_0$ satisfies that $\{x \in \NN : x >_{\L} x_0\}$ is finite, then we can contradict our assumption on the order by computing an infinite $\L$-decreasing sequence. Indeed, under our assumptions, there exists an infinite computable sequence $F^0 < F^1 < \dots$ of sets realizing $p_0$ and such that the sequence $(x_i)_{i \in \NN}$, where $x_i = \max_{\L} F^i$ is infinite and $\L$-decreasing. 

Such a sequence can be defined as follows: for every $i \in \NN$, assuming that $F_j$ has been defined for every $j < i$, the set $\bigcup_{j < i} \{x \in \NN : x >_{\L} x_j\}$ is finite, as it is equal to $\{x \in \NN : x >_{\L} x_{i-1}\}$ if $i > 0$ and to $\emptyset$ if $i = 0$, hence its complement is infinite and computable and we can apply the induction hypothesis on the induced order $(\{x \in \NN : \forall j < i, x <_{\L} x_j \wedge x > F^j\}, <_{\L})$ and computably find an instance of $p_0$ at a position $F^i$ in that ordering. By construction, $x_i = \max F^i <_\L x_j$ and $F^i > F^j$ for every $j < i$.

\end{proof}

\begin{corollary}
$\RCA_0 + \SepRT^2_2$ implies $\ADS$ over $\omega$-models.
\end{corollary}
\begin{proof}
By an immediate relativization of \Cref{prop:separable-trivial-ads}.
\end{proof}

\begin{proposition}\label[proposition]{prop:separable-implies-ads}
For every separable permutation~$p$, $\RCA_0 \vdash \RT_2^2(p) \to \ADS$.
\end{proposition}
\begin{proof}
The proof of \Cref{prop:separable-trivial-ads} is formalizable over~$\RCA_0$, except the induction on the structure of~$p$ which is too complex because $p$ might be non-standard. If however $p$ is standard, then no induction is required on the structure of~$p$, as there are only finitely many steps.
\end{proof}

Note that in particular, for every separable permutation~$p$, $\RT^2_2(p)$ implies the cohesiveness principle ($\COH$) and the $\Sigma^0_2$-bounding scheme ($\BSig_2$) over~$\RCA_0$, as $\ADS$ implies both statements (see Hirschfeldt and Shore~\cite[Proposition 2.10, Proposition 4.5]{hirschfeldt2007combinatorial}). We shall see a direct proof of the second statement in the more general case of all permutations in \Cref{prop:rt22p-implies-bsig2}.

\begin{corollary}
Let $p$ be a pattern without divergent irreducible sub-pattern.
Then $\RCA_0 \vdash \RT^2_2(p) \to \ADS$.
\end{corollary}
\begin{proof}
Since the non-transitivity patterns (see (\ref{eq:non-transitivity})) are divergent and irreducible, $p$ does not contain these as sub-patterns, hence $p$ is transitive, so codes for a permutation.
By \Cref{prop:separable-no-diverging-irreducible}, $p$ is separable, hence by \Cref{prop:separable-implies-ads}, $\RCA_0 \vdash \RT_2^2(p) \to \ADS$.
\end{proof}

The known proofs of the Erd\H{o}s-Moser theorem are by many means very similar to the ones of Ramsey's theorem for pairs and two colors, apart from the disjunctive nature of~$\RT^2_2$. The lack of disjunction in~$\EM$ was exploited by Lerman, Solomon and Towsner~\cite{lerman2013separating} and Patey~\cite{patey2017iterative} to prove that $\RCA_0 + \EM$ does not imply $\RT^2_2$ over $\omega$-models. It is natural to wonder whether there exist Ramsey-like statements of intermediate strength, that is, patterns~$p$ such that $\RT^2_2(p)$ is strictly in between $\RT^2_2$ and $\EM$. The notion of minimal cover formalizes the absence of intermediate strength.

\begin{definition}
Let $(P, \leq)$ be a partial order. An element~$a \in P$ is a \emph{minimal cover} of $b \in P$ if $a > b$ and for every $c \in P$ such that $a \geq c \geq b$, either $c = a$, or $c = b$. An element~$a \in P$ is a \emph{strong minimal cover} of~$b \in P$ if $a > b$ and for every~$c \in P$ such that $a > c$, then $b \geq c$.
\end{definition}

Thanks to \Cref{thm:characterization-sep-bose} and \Cref{prop:separable-implies-ads}, one obtains the following very useful trichotomy.

\begin{proposition}\label[proposition]{prop:trichotomy-ads-em-3142}
Let $p$ be a pattern. Then one of the following holds: 
\begin{itemize}
    \item[(a)] $\RCA_0 \vdash \RT^2_2(p) \to \ADS$;
    \item[(b)] $\RCA_0 \vdash \EM \to \RT^2_2(p)$;
    \item[(c)] $\RCA_0 \vdash \RT^2_2(1302) \to \RT^2_2(p)$.
\end{itemize}
\end{proposition}
\begin{proof}
If $p$ does not code for a permutation, it contains one of the two non-transitivity patterns as sub-pattern (see (\ref{eq:non-transitivity})). Thus, given a coloring $f : [\NN]^2 \to 2$, every infinite $f$-transitive set $f$-avoids~$p$, so $\RCA_0 \vdash \EM \to \RT^2_2(p)$.
If $p$ codes for a permutation, then by \Cref{thm:characterization-sep-bose}, either it is separable, or it contains any of $1302$ and $2031$ as sub-pattern. In the first case, by \Cref{prop:separable-implies-ads}, $\RCA_0 \vdash \RT^2_2(p) \to \ADS$. In the second case, either $\RCA_0 \vdash \RT^2_2(1302) \to \RT^2_2(p)$, or $\RCA_0 \vdash \RT^2_2(2031) \to \RT^2_2(p)$, but since $1302$ and $2031$ are dual, $\RT^2_2(1302)$ and $\RT^2_2(2031)$ are equivalent over~$\RCA_0$, hence $\RCA_0 \vdash \RT^2_2(1302) \to \RT^2_2(p)$.
\end{proof}

\begin{corollary}\label[corollary]{cor:em-minimal-cover}
$\RT^2_2$ is a minimal cover of~$\EM$ in the set of Ramsey-like statements, partially ordered by implication over~$\RCA_0$.
\end{corollary}
\begin{proof}
Let $p$ be an $\RT^2_2$-pattern such that $\RCA_0 \vdash \RT^2_2 \to \RT^2_2(p)$ and $\RCA_0 \vdash \RT^2_2(p) \to \EM$. By \Cref{prop:trichotomy-ads-em-3142}, we have three cases:
Case 1: $\RCA_0 \vdash \RT^2_2(p) \to \ADS$. Then, since $\RCA_0 \vdash \ADS \wedge \EM \to \RT^2_2$, $\RCA_0 \vdash \RT^2_2(p) \to \RT^2_2$.
Case 2: $\RCA_0 \vdash \EM \to \RT^2_2(p)$. In this case, we are done.
Case 3: $\RCA_0 \vdash \RT^2_2(1302) \to \RT^2_2(p)$. By Mimouni and Patey~\cite{mimouni2025ramseylike}, $\RCA_0 + \RT^2_2(1302) \nvdash \EM$, contradicting the hypothesis that $\RCA_0 \vdash \RT^2_2(p) \to \EM$.
\end{proof}

\section{Basis for separable permutations}\label[section]{sec:basis-sep-permutations}

The separable permutations form a relatively large family. In this section, we define canonical sub-families for separable permutations, the $k$-fractals for $k \in \NN$, which enjoy some universality properties: for every~$k \geq 2$, every separable permutation is a sub-pattern of a $k$-fractal. These families will play an essential role in \Cref{sec:sep-perm-measure} and \Cref{sec:2dnc-sep} to prove our main theorem.

\begin{definition}
A family $\P$ of patterns is a \emph{basis} for another family~$\Q$ of patterns
if every pattern~$q \in \Q$ is a sub-pattern of some pattern~$p \in \P$.
\end{definition}

Suppose $\P$ is a basis for $\Q$. Then the statement \qt{For every~$p \in \P$, every coloring avoiding $p$ satisfies a property} implies \qt{For every~$q \in \Q$, every coloring avoiding~$q$ satisfies the same property}.
Indeed, if a coloring does not satisfy the property, then every pattern in~$\P$ appears in it, and as $\P$ is a basis for~$\Q$, every pattern in~$\Q$ also appears.

\begin{definition}
The family $\F_k$ of \emph{$k$-fractals} is defined inductively as follows:
\begin{itemize}
    \item The $k$-fractal of dimension~0 is the trivial permutation $p : [1]^2 \to 2$. 
    \item Let $p_n$ be the $k$-fractal of dimension~$n$. The $k$-fractal of dimension~$n+1$
    is the permutation $p_{n+1}$ defined as the direct sum $p_n \oplus \dots \oplus p_n$ of $k$ copies of $p_n$  if $n$ is even, and as the skew sum $p_n \ominus \dots \ominus p_n$ of $k$ copies of $p_n$ if $n$ is odd. 
\end{itemize}
\end{definition}

Note that since the skew sum and the direct sum are associative, the previous definition is well-defined.

\begin{remark}
For a fixed coloring $f : [\NN]^2 \to 2$, by a small abuse of notation, we shall simply call \emph{$k$-fractal of dimension~$n$} any finite set $f$-realizing the $k$-fractal of dimension~$n$.
\end{remark}

\begin{example}
Let $f : [\NN]^2 \to 2$ be a coloring. The sets $f$-realizing the $k$-fractal of dimension~0 are the singletons.
The sets $f$-realizing the $k$-fractals of dimension~1 are the $f$-homogeneous sets for color~0 of cardinality~$k$. The sets $f$-realizing the $k$-fractals of dimension~2 are the finite sets of cardinality $k^2$ which can be partitionned into $k$ blocks $F_0 < F_1 < \dots F_{k-1}$, such that each~$F_i$ is $f$-homogeneous for color~0, and the pairs between any two blocks have color~1.
\end{example}

\begin{proposition}\label[proposition]{prop:fractal-basis-separable}
     For every $k \geq 2$, the family $\F_k$ is a basis for the family of separable permutations. 
\end{proposition}

\begin{proof}
Inductively on the structure of a separable permutation $p$, obtained as a series of direct and skew sums, we can show that there exists some $n$ such that $p$ is a sub-pattern of the $k$-fractal of dimension $n$. \\

This is the case for the trivial permutation of a single point. It is a sub-pattern of the $k$-fractal of dimension $0$.

If $p = p_0 \oplus p_1$, with $n_0,n_1$ such that $p_i$ a sub-pattern of the $k$-fractal of dimension $n_i$ for every $i < 2$. Let $n$ be any odd number greater than $\max(n_0,n_1)$. Then $p_0$ and $p_1$ are sub-pattern of the $k$-fractal of dimension $n-1$, hence $p = p_0 \oplus p_1$ is a sub-pattern of the $k$-fractal of dimension $n$ (since $k \geq 2$).

The case where $p$ is a skew sum of two separable permutations $p_0, p_1$ is similar, with~$n$ taken to be even.
\end{proof}

The following lemma can be seen as a reformulation of the partition lemma for bushy trees \cite[Lemma 2.7]{khan2017forcing} in the context of fractals. Indeed, the $k$-fractal of dimension~$n$ is isomorphic to a $k$-bushy tree of depth~$n$.

\begin{lemma}\label[lemma]{lem:partition-fractal}
    For every $a,b \in \NN$, every $n \in \NN$, and every 2-coloring of the vertices of the $(a + b - 1)$-fractal of dimension $n$, then, either the $a$-fractal of dimension $n$ is a homogeneous sub-pattern for the color $0$ or the $b$-fractal of dimension $n$ is a homogeneous sub-pattern for the color $1$.
\end{lemma}

\begin{proof}
By induction on $n$. The case $n = 0$ is clear.

Suppose the property holds for some $n \in \NN$. Assume without loss of generality that $n$ is even, the case $n$ odd being similar. Let $p_{n+1}$ be the $(a + b - 1)$-fractal of dimension $n+1$ and $f : p_{n+1} \to 2$ be a $2$-coloring. The pattern $p_{n+1}$ can be written as the direct sum $p_n \oplus \dots \oplus p_n$ of $a + b - 1$ copies of the $(a + b - 1)$-fractal of dimension $n$. For $i < a + b - 1$, denote by $p_n^i$ the $i$-th such copy. 

The coloring $f$ induces a coloring on each $p_n^i$, hence, by the induction hypothesis, $p_n^i$ either has the $a$-fractal of dimension $n$ as an $f$-homogeneous sub-pattern for the color $0$, or the $b$-fractal of dimension $n$ as an $f$-homogeneous sub-pattern for the color $1$. Define $c_i < 2$ to be equal to $0$ if we are in the first case and equal to $1$ in the second and denote by $q^i$ the sub-pattern in either cases.

By the finite pigeonhole principle, either $|\{i < a + b - 1 : c_i = 0\}| \geq a$ or $|\{i < a + b - 1 : c_i = 1\}| \geq b$. In the first case, take $a$ indexes $i_0 < \dots < i_{a-1}$ such that $c_{i_0} = \dots = c_{i_{a-1}} = 0$, and consider the direct sum $q_{i_0} \oplus \dots \oplus q_{i_{a-1}}$. This direct sum is the $a$-fractal of dimension $n+1$ and is $f$-homogeneous for the color $0$. The second case is similar.
\end{proof}

\section{Separable permutations and measure}\label[section]{sec:sep-perm-measure}

In this section, we prove that for every separable permutation~$p$, every computable instance of~$\coRT^2_2(p)$ admits a probabilistic solution. More precisely, we prove that the measure of oracles computing a solution is~1.
An analysis of the complexity of the argument using the tools of algorithmic randomness yields that every 2-random is sufficient to compute such a solution. The argument will be later refined in \Cref{sec:2dnc-sep} using diagonal non-computation, to obtain \Cref{maintheorem1}.
\bigskip

\noindent
\textbf{First example.}
As a warm-up, consider the statement $\coRT^2_2(120)$, where the permutation~120 corresponds to the following pattern:
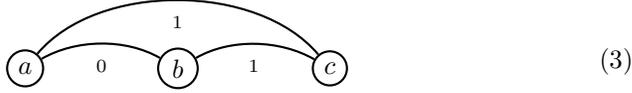
\vspace{1cm}
\begin{equation}\label{eq:perm120}
\patternthree{3}
\end{equation}
The statement was already studied by various authors under the names $\mathsf{SHER}$ (Dorais et al.~\cite{dorais2016uniform}), $\mathsf{TAC}$ (Conidis~\cite{conidistac}) $\mathsf{CAC\ for\ trees}$ and $\mathsf{SAC}$ (Cervelle, Gaudelier and Patey~\cite{cervelle2024reverse}). In particular, Cervelle, Gaudelier and Patey~\cite[Section 3]{cervelle2024reverse} proved that every computable instance of $\mathsf{SAC}$ admits a probabilistic solution. We give here the argument in the language of $\coRT^2_2(120)$ as it already contains the core ideas of the generalization to all separable permutations.

Let $f : [\NN]^2 \to 2$ be a computable coloring avoiding the permutation~120. 
Given an infinite set~$X \subseteq \NN$, we say that an element~$x \in \NN$ is \emph{$X$-bad} if for all but finitely many~$y \in X$, $f(x, y) = 1$. Otherwise, $x$ is \emph{$X$-good}.
Since $f$ avoids~120, then for every infinite set~$X \subseteq \NN$, every finite $f$-homogeneous set~$F \subseteq \NN$ for color~0 contains at most one $X$-bad element. Thus, if one picks an element~$x \in F$ at random, the chances that it is $X$-bad is at most $\frac{1}{\card F}$.

Consider the following probabilistic algorithm to compute an infinite $f$-homogeneous set for color~0: Start with the computable Mathias condition $(\emptyset, \NN)$.
At step~$s$, assume we have defined a computable Mathias condition $(\sigma_s, X_s)$, where $\card \sigma_s = s$ and $\forall x \in \sigma_s\ \forall y \in X_s\ f(x, y) = 0$. Choose a finite $f$-homogeneous set~$F_s \subseteq X_s$ for color~0 of size~$2^{s+2}$, and pick an element~$x \in F_s$ at random. Let $\sigma_{s+1} = \sigma_s \cup \{x\}$ and $X_{s+1} = \{ y \in X_s : y > x \wedge f(x, y) = 0 \}$ and go to the next step.

There are two reasons why the algorithm could fail. First, given a computable Mathias condition $(\sigma_s, X_s)$, there might be no $f$-homogeneous subset of~$X_s$ for color~0 of size~$2^{s+2}$. In this case, $f$ has a very degenerate behavior in~$X$, and by a classical argument, there exists an infinite computable $f$-homogeneous set for color~1 (see \Cref{lem:unbalanced-ramsey}). If there is such a set, then we are done, so assume this issue does not arise. The other reason why this algorithm could fail is if the set $X_{s+1} = \{ y \in X_s : y > x \wedge f(x, y) = 0 \}$ is finite, in other words if $x$ is $X_s$-bad. The chances that the algorithm fails at step~$s$ is then at most $\frac{1}{\card F_s} = 2^{-s-2}$, so the chances that the algorithm fails at some point is at most $\sum_s 2^{-s-2} = \frac{1}{2}$. One can parameterize the algorithm by a constant~$k \in \NN$, and search for sets $F_s$ of cardinality $2^{s+1+k}$, in which case the probability of failure is at most $2^{-k}$.


By coding the choices within the blocks $\{F_s\}_{s \in \NN}$ into chunks of an infinite binary sequence, that is, decomposing any $R \in \cs$ into an infinite sequence $\sigma^R_0 \cdot \sigma^R_1 \cdot \dots$ where $\sigma^R_s$ is a string of length~$2^{s+k+1}$ coding for the choice of an element in~$F_s$, the measure of oracles $R \in \cs$ such that this decomposition yields a valid algorithm is at least~$1-2^{-k}$.  This tends to~1 when $k$ tends to infinity, so the measure of oracles computing a solution is~1.

One can go one step further, and analyse the amout of randomness needed to make this computation work, using algorithmic randomness. See Nies~\cite{nies2009computability} or Downey and Hirschfeldt~\cite{downey2010algorithmic} for a good introduction to the topic. In what follows, $\mu$ denotes the Lebesgue measure over the Cantor space.

\begin{definition}
A \emph{$\Sigma^0_1(Z)$-test} (or Martin-L\"of test relative to~$Z$) is a uniform sequence $(\U_n)_{n \in \NN}$ of $\Sigma^0_1(Z)$ classes such that for every~$n \in \NN$, $\mu(\U_n) \leq 2^{-n}$.
A set~$X$ is \emph{$Z$-random} if $X \not \in \bigcap_n \U_n$ for every $\Sigma^0_1(Z)$-test $(\U_n)_{n \in \NN}$.
\end{definition}

We shall call \emph{1-random} and \emph{2-random}\footnote{The notion of 2-random set is actually defined in terms of $\Sigma^0_2$-tests instead of $\Sigma^0_1(\emptyset')$-tests. However, for every $\Sigma^0_2$ class of positive measure~$\U \subseteq \cs$ and every~$\epsilon > 0$, one can uniformly compute an index of a $\Sigma^0_1(\emptyset')$ class $\V \supseteq \U$ such that $\mu(\V) < \mu(\U) + \epsilon$, so the two notions coincide (see Kautz~\cite{kautz1991degrees}).} any $\emptyset$-random and $\emptyset'$-random, respectively. These notions, and in particular 2-randomness, were studied in the framework of reverse mathematics by various authors (see \cite{conidis2013random,nies2020randomness,belanger2021where}).

Being $X$-bad is a $\Sigma^0_1(X')$ property, so given~$k \in \NN$, the class $\U_k$ of oracles for which the algorithm with parameter~$k$ fails is $\Sigma^0_1(\emptyset')$ uniformly in~$k$, and has measure at most~$2^{-k}$. The collection $(\U_k)_{k \in \NN}$ is therefore a $\Sigma^0_1(\emptyset')$-test, so for every 2-random~$R$, $R \not \in \bigcap_k \U_k$, so there is some~$k \in \NN$ for which the $k$-parameter algorithm succeeds following the choices of~$R$. Thus, every 2-random computes a solution.

\bigskip
\noindent
\textbf{General case.} Let us now consider arbitrary separable permutations. The fractals forming a basis for separable permutations, we shall restrict ourselves to this sub-class.
We start with a technical lemma which can be considered as folklore. It admits several proofs, using various levels of induction. However, its exact reverse mathematical strength is currently unknown (see Frittaion, personal communication).

\begin{lemma}[Folklore]\label[lemma]{lem:unbalanced-ramsey}
    Let $f : [\NN]^2 \to 2$ be a computable coloring and $k \in \NN$ such that no clique of size $k$ is $f$-homogeneous for the color $0$. Then there exists some infinite computable $f$-homogeneous set for the color $1$.   
\end{lemma}
\begin{proof}[Belanger's proof over~$\ISig_2$ (personal communication)]
Let $T \subseteq \NN^{<\NN}$ be the computable tree such that
\begin{itemize}
    \item[(1)] every $x \in \NN$ appears in some node~$\sigma \in T$;
    \item[(2)] every $\sigma \in T$ is $f$-homogeneous for color~0;
    \item[(3)] for every $\sigma \in T$ and every~$x, y \in \NN$ such that $\sigma \cdot x, \sigma \cdot y \in T$, $f(x, y) = 1$.
\end{itemize}
By (2) and by assumption, $T$ has finite depth, so by (1) and $\BSig_2$, there is some level~$\ell_0 \in \NN$ such that $\{ \sigma \in T : |\sigma| = \ell_0 \}$ is infinite. By $\ISig_2$, there is a least such level $\ell \leq \ell_0$. Since $T$ has only one node at level~0, namely $\langle \rangle$, $\ell > 0$. By minimality of~$\ell$, $\{ \sigma \in T : |\sigma| = \ell-1 \}$ is finite, so by $\BSig_2$, there is some~$\sigma \in T$ of length~$\ell-1$ with infinitely many immediate children. The set $\{ x \in \NN : \sigma \cdot x \in T \}$ is therefore infinite, and by (3), it is $f$-homogeneous for color~1.
\end{proof}

Given an infinite set~$X \subseteq \NN$, we say that an element~$x \in \NN$ has \emph{limit $c$ in~$X$} if for all but finitely many~$y \in X$, $f(x, y) = c$. The following proposition is the probabilistic core of the argument, and generalizes the case $\coRT^2_2(120)$ by considering blocks which have an arbitrarily small ratio of bad elements instead of a fixed number.

\begin{proposition}\label[proposition]{prop:2random-computes-homogeneous}
Let $f : [\NN]^2 \to 2$ be a computable coloring, $c < 2$ be a color, and $\varphi(k, F)$ be a $\Sigma^0_1$ formula, such that, for every infinite computable set~$X \subseteq \NN$:
\begin{itemize}
    \item[(1)] For every $k > 0$, there is a finite set~$F \subseteq X$ such that $\varphi(k, F)$ holds; 
    \item[(2)] For every~$k > 0$ and every finite set~$F \subseteq \NN$ such that $\varphi(k, F)$ holds, there are at most $\frac{\card F}{k}$ elements $x \in F$ having limit $1 - c$ in $X$.
    
    
\end{itemize}
Then every $2$-random set computes an infinite $f$-homogeneous set for the color~$c$.
\end{proposition}

\begin{proof}
Fix such a coloring $f$, a color $c$ and a formula $\varphi(k,F)$. Consider the $\Sigma^0_1$-formula 
$$
\psi(k,F) \equiv \exists F' \exists x, F \subseteq F' \wedge \card F = 2^x \wedge \card F' < 2^{x+1} \wedge \varphi(2k,F')
$$
Then $\psi(k,F)$ satisfies the same hypothesis as $\varphi(k,F)$, but with the added property that if $\psi(k,F)$ holds, then the cardinality of $F$ is a power of $2$. Thus, without loss of generality, we can assume that $\varphi(k,F)$ also has this property. 

A finite increasing sequence $\sigma \in \NN^{<\NN}$ is \emph{valid} if the set $X_\sigma = \{ z \in \NN : \forall i < |\sigma|\ f(\sigma(i), z) = c \}$ is infinite. Let $V \subseteq \NN^{<\NN}$ be the $\Pi^0_2$ set of valid sequences.

Fix some $n \in \NN$ and consider some increasing computable sequence $2 \leq u_0 < u_1 < \dots$ such that $\sum_{i = 0}^{\infty} \frac{1}{u_i} < 2^{-n}$. 
For every valid sequence $\sigma \in V$, a sequence of finite sets $F_0 < F_1 < \dots < F_{|\sigma|}$ will be said to be a \emph{$\sigma$-sequence} if it satisfies the following properties:
\begin{itemize}
    \item For every $i < |\sigma|$, $\sigma(i) \in F_i$ 
    \item $\varphi(u_i,F_i)$ holds for every $i \leq |\sigma|$.
    \item For every $i < j \leq |\sigma|$ and every $y \in F_j$, $f(\sigma(i),y) = c$.
\end{itemize}
Note that given $\sigma \in V$, being a $\sigma$-sequence is a $\Sigma^0_1$-predicate.

An $\epsilon$-sequence $\langle F_0 \rangle$ is \emph{minimal} if $F_0$ is the least set found in a computable search. A $(\sigma \cdot x)$-sequence $\langle F_0, \dots, F_{|\sigma|}, F_{|\sigma|+1}\rangle$  is \emph{minimal} if the $\sigma$-sequence $\langle F_0, \dots, F_{|\sigma|}\rangle$ is minimal and $F_{|\sigma|+1}$ is the least set found in a computable search. Thus, every $\sigma \in V$ admits at most one minimal $\sigma$-sequence. Let $T_n \subseteq V$ be the $\Pi^0_2$ set of valid~$\sigma \in V$ admitting a minimal $\sigma$-sequence.
\smallskip

\textbf{Claim 1}: The set $T_n$ is an infinite, finitely branching tree. Indeed, the property of being valid is closed under prefix, as well as the property of admitting a minimal $\sigma$-sequence, so the set $T_n$ is a tree. Let $\sigma$ be an element of $T_n$, and let $\langle F_0, \dots, F_{|\sigma|}\rangle$ be its corresponding minimal $\sigma$-sequence, then, every $x$ such that $\sigma \cdot x \in T_n$ is such that $x \in F_{|\sigma|}$, so there are finitely many such $x$, and $T_n$ is finitely branching.

Let $\sigma$ be an element of $T_n$ with minimal sequence $\langle F_0, \dots, F_{|\sigma|}\rangle$. Since $\sigma$ is valid, the set $X_{\sigma}$ is infinite, and since $\varphi(u_{|\sigma|}, F_{|\sigma|})$ holds, there are at most $\frac{\card F_{|\sigma|}}{u_{|\sigma|}}$ elements $x \in F_{|\sigma|}$ such that the set $\{ y \in X_{\sigma} : f(x, y) = c \} = X_{\sigma \cdot x}$ is finite, hence, for all the other $x \in F_{|\sigma|}$, $\sigma \cdot x$ will be valid.

Pick some $x \in F_{|\sigma|}$ such that $X_{\sigma \cdot x}$ is infinite. By the properties of $\varphi(k,F)$, there exists some $F_{|\sigma| + 1} \subseteq X_{\sigma \cdot x} \setminus \{0, \dots, \max F_{|\sigma|}\}$ such that $\varphi(u_{|\sigma|+1}, F_{|\sigma| + 1})$ holds. By definition of $X_{\sigma \cdot x}$, for every $i < |\sigma|$ and every $y \in F_{|\sigma| + 1}$, $f(\sigma(i),y) = c$ and $f(x,y) = c$, hence $\langle F_0, \dots, F_{|\sigma| + 1}\rangle$ is a $(\sigma \cdot x)$-sequence. Since $\langle F_0, \dots, F_{|\sigma|}\rangle$ is minimal, if $F_{|\sigma| + 1}$ is chosen to be the least such set found in a computable search, then $\langle F_0, \dots, F_{|\sigma| + 1}\rangle$ will be a minimal $(\sigma \cdot x)$-sequence, and $\sigma \cdot x \in T_n$. Therefore, every element of $T_n$ has at least one descendant, and since $T_n$ is not empty (as it contains the empty sequence), $T_n$ is an infinite tree.
This proves Claim~1.

\smallskip
\textbf{Claim 2}: For every $\sigma \in T_n$ with minimal $\sigma$-sequence $\langle F_0, \dots, F_{|\sigma|}\rangle$, $\sigma \cdot x \in T_n$ for at least $\frac{(u_{|\sigma|}-1)\card F_{|\sigma|}}{u_{|\sigma|}}$ elements $x \in F_{|\sigma|}$. Indeed, we saw that there are at least $\card F_{|\sigma|} - \frac{\card F_{|\sigma|}}{u_{|\sigma|}}$ elements of $F_{|\sigma|}$ such that $X_{\sigma \cdot x}$ is infinite, and that for every such element $\sigma \cdot x \in T_n$.
This proves Claim~2.
\smallskip

An infinite branch $X = \{x_0 < x_1 < \dots \}$ in $T_n$ corresponds to a unique sequence $F_0 < F_1 < \dots$ such that for every $i \in \NN$, $\langle F_0, \dots, F_{i} \rangle$ is the minimal $(X \uh i)$-sequence. Such a branch can be encoded as an element $(b_0,b_1, \dots) \in 2^{\NN}$ of the Cantor space as follows: by our added assumption on $\varphi(k,F)$, there exists $(a_i)_{i \in \NN}$ such that $\card F_0 = 2^{a_0}$ and $\card F_{i+1} = 2^{a_{i+1} - a_i}$ for every $i \in \NN$, then let $b_0, \dots, b_{a_0 - 1}$ be the binary encoding of the position of $x_0$ in $F_0$, and let $b_{a_i}, \dots, b_{a_{i+1} - 1}$ be the binary encoding of the position of $x_{i+1}$ in $F_{i+1}$ for every $i \in \NN$.

Therefore, the set of infinite branches in $T_n$ can be described as a $\Pi_2^0$ subclass of the Cantor space $2^{\NN}$. From an element $(b_0, b_1 \dots) \in 2^{\NN}$ of the Cantor space, we can computably try to see if it encodes some infinite branch $X = \{x_0 < x_1 < \dots \}$ in $T_n$ as follows: there is only one set $F_0$ (which can computably be found) such that $\langle F_0 \rangle$ is the minimal $\epsilon$-sequence, thus, we can recover $x_0$ by looking at the element of $F_0$ in position encoded in binary by $b_0, \dots, b_{(\log |F_0|) - 1}$, then, there exists at most one $F_1$ such that $\langle F_0, F_1\rangle$ is the minimal $x_0$-sequence and we can similarly retrieve $x_1$, if $(b_0, b_1 \dots)$ was indeed the encoding of an infinite branch of $T_n$, then the decoding procedure is able to continue like this to retrieve the full branch. 

If the element $(b_0, b_1 \dots) \in 2^{\NN}$ is not the encoding of an infinite branch in $T_n$, then the decoding procedure will eventually stay stuck, not being able to find a set $F_i$ to continue the decoding. This happens if for some $j < i$, the element $x_{j}$ previously decoded is such that $x_0 \dots x_{j}$ is not valid (in that case, since $X_{x_0 \cdots x_j}$ is finite, there will be at stage after which no $F_i$ to continue the decoding can be found). By Claim 2, if $x_0 \cdots x_{j-1}$ is valid, then there are less than $\frac{\card F_j}{u_j}$ elements $x \in \card F_j$ such that $x_0\cdots x_{j-1} x$ is not valid, hence the probability of that happening at stage $j$ of the decoding is less than $\frac{1}{u_j}$, and the probability of that happening at any stage less than $\sum_{i=0}^{\infty} \frac{1}{u_i} < 2^{-n}$. Therefore, the measure of the set of infinite branches of $T_n$ is greater than $1 - 2^{-n}$ with this encoding. 

For every 2-random set $Z$, there will be some value $n$ such that $Z$ is the encoding of an infinite path $x_0 < x_1 < \dots$ through the corresponding tree $T_n$. Since our decoding procedure was computable, $Z$ computes the set $\{x_0, x_1, \dots\}$, which is an infinite $f$-homogeneous set for the color $c$.
\end{proof}

The following proposition is the desired result, restricted to fractals. The theorem about separable permutations follows directly from it, since the fractals form a basis for separable permutations.

\begin{proposition}\label[proposition]{prop:probabilistic-fractal-avoiding}
    Let $f : [\NN]^2 \to 2$ be a computable coloring. For every $n, k \in \NN$, if $f$ avoids the $k$-fractal of dimension $n$, then every 2-random set computes an infinite $f$-homogeneous set.
\end{proposition}

\begin{proof}
Fix a computable coloring $f : [\NN]^2 \to 2$. By induction on~$n$, we will show that for any $k \in \NN$ and every infinite computable set~$D \subseteq \NN$ $f$-avoiding the $k$-fractal of dimension~$n$, every $2$-random set computes an infinite $f$-homogeneous subset of~$D$.

The case $n = 0$ is vacuously true, as no set $f$-avoids the singleton pattern.

In the case $n = 1$, if some infinite computable set~$D$ $f$-avoids the $k$-fractal of dimension~1, then there is no $f$-homogeneous subset of~$D$ for color~0 of size~$k$. Then, by \Cref{lem:unbalanced-ramsey}, there exists an infinite computable $f$-homogeneous subset of~$D$ for the color $1$.


Assume the property to be true for some $n \in \NN$. Fix an infinite computable set~$D$ such that some 2-random set does not compute any infinite $f$-homogeneous subset of~$D$. 
By our induction hypothesis, for every $k' \in \NN$, and every infinite computable set $X \subseteq D$, there is some set $S \subseteq X$ $f$-realizing the $k'$-fractal of dimension~$n$.

Fix a size $k \in \NN$ and let $c$ be equal to $1$ if $n$ is even and $0$ if $n$ is odd. Consider the two following cases:
\smallskip


\textbf{Case 1:} There exists some infinite computable set $G \subseteq D$, such that, for every infinite computable subset $H \subseteq G$, every $S \subseteq H$ $f$-realizing the $k$-fractal of dimension $n$ has at least one element which does not have limit $1-c$ in $H$. 

 

Fix some infinite computable subset $H \subseteq G$, some $\ell \in \NN$ and some $S \subseteq G$ that $f$-realizes the $(k\times \ell - 1)$-fractal of dimension $n$ (such an $S$ exists by our induction hypothesis). We claim that out of the $(k \times \ell - 1)^n$ elements in $S$, at least $(k \times (\ell - 1))^n$ does not have limit $1 - c$ in $H$ for $f$. Indeed, by \Cref{lem:partition-fractal}, either there exists a subset of $S \subseteq G$ $f$-realizing the $k$-fractal of dimension $n$ such that every element has limit behavior $1 - c$ in $H$ for $f$, or there exists a subset of $S \subseteq G$ $f$-realizing the $(k \times (\ell - 1))$-fractal of dimension $n$ such that no element has limit behavior $1 - c$ in $H$ for $f$. The first case being impossible by our assumption, the second case holds, and since the cardinality of the $(k \times (\ell - 1))$-fractal of dimension $n$ is $(k\times (\ell - 1))^n$, we are done.

Since $\lim_{\ell \to \infty} \frac{(k\times (\ell - 1))^{n}}{(k\times \ell - 1)^{n}} = 1$, for every $k$ there exists a bound $\ell_k$ such that every finite set $S \subseteq G$ $f$-realizing the $(k\times \ell_k - 1)$-fractal of dimension $n$ will have at most $\frac{\card S}{k}$ elements having limit $1-c$ for $f$ in any infinite computable subset $H \subseteq G$.

Say $D = \{ x_0 < x_1 < \dots \}$.
Thus, for $\varphi(k,F)$ the $\Sigma_1^0$ formula stating that $\{ x_s : s \in F \}$ is a finite set that $f$-realizes the $(k\times \ell_k - 1)$-fractal of dimension $n$, we get by \Cref{prop:2random-computes-homogeneous} that every $2$-random computes an infinite $f$-homogeneous set for the color $c$.
\smallskip

\textbf{Case 2:} For every infinite computable set $G \subseteq D$, there exists some infinite computable subset $H \subseteq G$ and some $S \subseteq H$ $f$-realizing the $k$-fractal of dimension $n$ and with all its elements having limit $1-c$ in $H$.

In that case, let $H_0 \subseteq D$ be an infinite computable set and $S_0 \subseteq H_0$ be a finite set $f$-realizing the $k$-fractal of dimension $n$ such that all its elements have limit $1-c$ in $H_0$. By our assumption, such a pair exists. Let $G_1 \subseteq H_0$ be such that every element of $S_0$ has reached its limit in $G_1$. Then, there exists an infinite computable set $H_1 \subseteq G_1$ and a finite set $S_1 \subseteq H_1$ $f$-realizing the $k$ fractal of dimension $n$ such that all its element have limit $1-c$ in $H_1$. Pick such a pair $H_1$, $F_1$, and iterate the construction to obtain a sequence $F_0 < \dots < F_{k-1}$ such that $F_0 \cup \dots \cup F_{k-1}$ $f$-realizes the $k$-fractal of dimension $n+1$.
\end{proof}

We are now ready to prove the main theorem of this section.

\begin{theorem}\label[theorem]{thm:avoid-separable-randomness}
    For every separable permutation~$p$ and every computable instance of $\coRT^2_2(p)$, every 2-random computes a solution.
\end{theorem}

\begin{proof}
Let $f : [\NN]^2 \to 2$ be a computable coloring avoiding~$p$.
By \Cref{prop:fractal-basis-separable}, there exists some $k,n \in \NN$ such that it avoids the $k$-fractal of dimension $n$. Then by \Cref{prop:probabilistic-fractal-avoiding} every 2-random computes an infinite $f$-homogeneous set.
\end{proof}

\section{Separable permutations and $\DNCs{2}$}\label[section]{sec:2dnc-sep}

We now refine the probabilistic argument of \Cref{sec:sep-perm-measure} to obtain better bounds in terms of relativized diagonal non-computation. Intuitively, it would be sufficient to have a partial choice function instead of a total one, to find good elements in the blocks of sets and obtain an infinite homogeneous set. One therefore does not need to have access to all the bits of the 2-random sequence. It is therefore natural to consider the computational power corresponding to finding infinitely many bits of a random sequence, which is exactly diagonal non-computation.
\bigskip

\noindent
\textbf{DNC functions}. Historically, DNC functions are defined as follows:

\begin{definition}
A function $f : \NN \to \NN$ is \emph{diagonally non-$Z$-computable} or $Z$-DNC if for every~$e$, $f(e) \neq \Phi^Z_e(e)$. A Turing degree is \emph{$Z$-DNC} if it bounds a $Z$-DNC function.
\end{definition}

The diagonally non-computable degrees have many characterizations, using various paradigms, such as fixpoint-free functions (\cite{jockusch1989recursively}), infinite subsets of random sequences (\cite{kjos2009infinite,greenberg2009lowness}, see \cite[Theorem 8.10.2]{downey2010algorithmic}) sequences of high Kolmogorov complexity (\cite{kjos2011kolmogorov}). The following characterization, in terms of uniform avoidance of finite $\Sigma^0_1(Z)$ sets with bounded size, is arguably the most useful:

\begin{definition}
A function $g : \NN^2 \to \NN$ is \emph{$Z$-escaping} if for every $e, k \in \NN$, if $\card W_e^Z \leq k$,
    then $g(e, k) \not \in W_e^Z$.
\end{definition}

The following lemma can be considered as folklore. We reprove it for the sake of completion.

\begin{lemma}[{Folklore (see \cite[Lemma 3.3]{bienvenu2017logical})}]\label[lemma]{lem:equiv-dnc-avoid}
Fix a set $Z$ and a Turing degree~$\dbf$. The following are equivalent:
\begin{itemize}
    \item[(1)] $\dbf$ bounds a $Z$-DNC function;
    \item[(2)] $\dbf$ bounds a $Z$-escaping function.
\end{itemize}
\end{lemma}
\begin{proof}
$(1) \to (2)$. Given a $Z$-DNC function~$f$ and $e, k \in \NN$, define $g(e, k)$ as follows:
For every~$i \leq k$, let $\Phi_{a_i}^Z$ be the function which, on every input, searches for the $i$th element~$x_i$ of~$W_e^Z$ in order of apparition, if it exists, and interprets $x_i$ as a $(k+1)$-tuple $\langle x^0_i, \dots, x^k_i \rangle$ and outputs~$x^i_i$. Let $g(e, k) = \langle f(a_0), \dots, f(a_k)\rangle$.
Suppose for the contradiction that $\card W_e^Z \leq k$ but $g(e, k) = x_i \in W_e^Z$ for some~$i \leq k$. Then $\Phi_{a_i}^Z(a_i) = x_i^i = f(a_i)$, contradicting the fact that $f$ is $Z$-DNC.

$(2) \to (1)$. Given a $Z$-escaping function~$g$, and $e \in \NN$, define $f(e)$ as follows:
Let $W_a^Z$ be the $Z$-c.e.\ set whose unique element is $\Phi^Z_e(e)$, if it exists, and $W_a^Z = \emptyset$ otherwise. Let $f(e) = g(a, 1)$. By construction, $\card W_a^Z \leq 1$, so $g(a, 1) \not \in W_a^Z$, hence $f(e) = g(a, 1) \neq \Phi^Z_e(e)$.
\end{proof}

One can refine the previous lemma to require the $Z$-escaping function to output its value within an infinite computable set~$X \subseteq \NN$, by transforming any $Z$-c.e.\ set $W_e^Z \subseteq X$ into the set of positions of its elements in~$X$. 

\bigskip
\noindent
\textbf{Computably bounded escaping functions}.
The situation becomes more problematic if one wants the $Z$-escaping function $g$ to output its value within a finite set~$F \subseteq \NN$. One would require $g$ to be $h$-bounded, for a computable function~$h$. By an analysis of the proof of \Cref{lem:equiv-dnc-avoid}, the existence of a computably bounded $Z$-escaping function is equivalent to the existence of a computably bounded $Z$-DNC function (with a different computable bound). Such functions are still probabilistic objects for sufficiently fast-growing bounds, in that every $Z$-random computes a computably bounded $Z$-DNC function. They have been studied by multiple authors~\cite{beros2024imunity,greenberg2011diagonally}, and in particular by Khan and Miller~\cite{khan2017forcing}, who proved the existence, for every~$Z$, of $Z$-DNC functions which do not compute any computably bounded DNC function, and a fortiori any 2-random set.

Although strictly weaker than 2-randomness, computing solutions from any computably bounded escaping function is not satisfactory for the following reason: The underlying goal is to prove that for every separable permutation~$p$, $\RT^2_2(p)$ implies $\coRT^2_2(p)$ over $\omega$-models, to deduce that $\RT^2_2(p)$ implies $\RT^2_2$. Mimouni and Patey~\cite{mimouni2025ramseylike} constructed a computable coloring $f : [\NN]^2 \to 2$ such that any infinite set $f$-avoiding any pattern~$p$ (not only separable permutations) computes a $\emptyset'$-DNC function. It follows that for every pattern~$p$, $\RT^2_2(p)$ implies the existence of a $\emptyset'$-DNC function. There is however no hope of improving this bound by showing that any pattern avoiding~$p$ computes a computably bounded $\emptyset'$-DNC function, as Liu~\cite{liu2015cone} proved that every computable coloring $f : [\NN]^2 \to 2$ admits an infinite $f$-homogeneous set which does not compute any computably bounded DNC function.

\bigskip
\noindent
\textbf{Partial escaping functions.}
Thankfully, we shall see that any $\emptyset'$-DNC function (or equivalently any $\emptyset'$-escaping function) computes a partial, computably bounded $\emptyset'$-escaping function, in the following sense:

\begin{lemma}\label[lemma]{lem:DNC-avoid-infinite}
    Let $B \subseteq \NN$ be a $\Sigma^0_2$ set of \qt{bad elements}, $(F_i)_{i \in \NN}$ be a computable family of subsets of $\NN$ and $k \in \NN$ such that: $\card F_i \cap B \leq k$ and $\card F_i \geq i$ for every $i$. Then, every $\emptyset'$-DNC degree $\dbf$ computes an infinite subset of $\bigcup_i F_i \setminus B$. 
\end{lemma}

\begin{proof}
Fix such $B \subseteq \NN$, $(F_i)_{i \in \NN}$, $k \in \NN$ and $\dbf$. We have two cases.
\smallskip

\textbf{Case 1:} $\dbf \geq_T \emptyset''$. Then $\dbf \geq_T B$ and therefore $\dbf \geq_T \bigcup_i F_i \setminus B$.

\smallskip

\textbf{Case 2:} $\dbf \not \geq_T \emptyset''$. Then, there exists some $Z \leq_T \emptyset'$ such that $\dbf \geq_T Z$ and $\dbf \not \geq_T Z'$. Indeed, if $\dbf \geq_T \emptyset'$, take $Z = \emptyset'$ and otherwise take $Z = \emptyset$.

Let $g \leq_T \dbf$ be the $\emptyset'$-escaping function of \Cref{lem:equiv-dnc-avoid}(2). 
Let $\mu_s$ be the standard left-c.e.\ approximation of the modulus $\mu$ of $Z'$, that is, $\mu_s(x)$ is the least~$t \leq s$ such that $Z'_t \uh x = Z'_s \uh x$. For every $x \in \NN$, the sequence $(\mu_s(x))_{s \in \NN}$ can only change values at most $x$ times (every time a program with Turing index less than $x$ stops), and let $m_0^x < \dots < m_i^x (= \mu(x))$ be those values for some $i < x$ which depends on $x$. Then, as $\emptyset' \geq Z$, let $e : \NN \to \NN$ be a computable function such that $W^{\emptyset'}_{e(x)}$ is the $\emptyset'$-enumerable set containing $\{0, \dots, x-1\}$ and the position in $F_{m_j^x}$ of the elements of $F_{m_j^x} \cap B$ for every $j < i$. 

Note that $\card W^{\emptyset'}_{e(x)} \leq x \times (k+1)$, hence, by hypothesis on $g$, letting $p_x = g(e(x), x \times (k+1))$, the $\dbf$-computable sequence $(p_x)_{x \in \NN}$ is such that $p_x \notin W^{\emptyset'}_{e(x)}$ for every $x$. In particular, $p_x \geq x$ for every~$x$ as $\{0, \dots, x-1\} \subseteq W^{\emptyset'}_{e(x)}$. Since $\dbf \not \geq_T Z'$, the sequence $(p_x)_{x \in \NN}$ does not dominate $\mu$, hence, for infinitely many $x$, we can $\dbf$-computably find some stage $s$ such that $\mu_s(x) \geq p_x$ and if $s$ is chosen to be minimal, this ensures that the element in position $p_x$ of $F_{\mu_s(x)}$ does not belong to~$B$. Let $S$ be this set of elements. Then $S \subseteq \bigcup_i F_i \setminus B$, and since~$p_x \geq x$, then $S$ is infinite.


\end{proof}

Consider for instance a stable computable version of $\coRT^2_2(120)$, that is, a computable coloring $f : [\NN]^2 \to 2$ such that for every~$x \in \NN$, $\lim_y f(x, y)$ exists, and which $f$-avoids the permutation~120. As explained in \Cref{sec:sep-perm-measure}, for every finite $f$-homogeneous set~$F \subseteq \NN$ for color~0, at most one element has limit~1. If we let $B = \{ x : \lim_y f(x, y) = 1 \}$ and $(F_n)_{n \in \NN}$ be a sequence of $f$-homogeneous sets for color~0 such that $\card F_n = n$ (if no such sequence exists, then by \Cref{lem:unbalanced-ramsey} there is a computable $f$-homogeneous set and we are done), then by \Cref{lem:DNC-avoid-infinite}, any $\emptyset'$-DNC degree bounds an infinite subset of $A = \{ x : \lim_y f(x, y) = 0 \}$, hence computes an infinite $f$-homogeneous set. One can therefore reprove the following theorem.

\begin{theorem}[Cervelle, Gaudelier and Patey~\cite{cervelle2024reverse}]\label[theorem]{thm:rt22-120-omega}
Every $\omega$-model of~$\RT^2_2(120)$ is a model of~$\RT^2_2$.
\end{theorem}
\begin{proof}
Let $\M$ be an $\omega$-model of~$\RT^2_2(120)$.
Given a coloring $f : [\NN]^2 \to 2$ in~$\M$, since $\RCA_0 \vdash \RT^2_2(120) \to \COH$ (see \Cref{prop:separable-implies-ads}), there is an infinite subset~$X_0 \subseteq \NN$ in~$\M$ on which $f$ is stable. Apply $\RT^2_2(120)$ to obtain an infinite subset $X_1 \subseteq X_0$ in~$\M$ which $f$-avoids the permutation~120. By Mimouni and Patey~\cite{mimouni2025ramseylike}, there is an~$X'_1$-DNC function~$h$ in~$\M$. By the argument above, $h \oplus X_1$ computes an infinite $f$-homogeneous set~$H$, which is therefore in~$\M$.
\end{proof}

Note that in this proof, three applications of $\RT^2_2(120)$ are used to prove $\RT^2_2$ in \Cref{thm:rt22-120-omega}.

\begin{remark}
The argument in \Cref{lem:DNC-avoid-infinite} admits many degrees of freedom, opening the door to many generalizations.
First, the notion of \qt{bad element} is absolute, in the sense that it does not depend on the previous choices of elements, while with non-stable colorings, one uses the notion of $X$-bad which depends on the reservoir~$X$, which itself depends on the choices of the previous elements.
Second, \Cref{lem:DNC-avoid-infinite} only states the existence of an infinite subset $S$ of $\bigcup_i F_i \setminus B$, but one could require that for some~$k \in \NN$, the set $S$ intersects simultaneously $F_i, F_{i+1}, \dots, F_{i+k}$ for infinitely many~$i \in \NN$.
We shall use both generalizations when dealing with arbitrary separable permutations, through \Cref{prop:2dnc-computes-homogeneous}.
\end{remark}

\bigskip
\noindent
\textbf{General case.} We now turn to the study of $\coRT^2_2(p)$ for arbitrary separable permutations~$p$.
As in \Cref{sec:sep-perm-measure}, we shall first prove the theorem for fractals, and deduce the general case since they form a basis for separable permutations.
By \Cref{prop:separable-implies-ads}, $\RT^2_2(p)$ implies $\COH$ for every separable permutation~$p$, so it would be sufficient to consider only stable instances of~$\coRT^2_2(p)$, as in \Cref{thm:rt22-120-omega}. We however prove it for every computable instance of~$\coRT^2_2(p)$.
It follows that $\coRT^2_2(p)$ is computably reducible to~$\RT^2_2(p)$, that is, every $\coRT^2_2(p)$-instance~$f$ computes an $\RT^2_2(p)$-instance~$g$, such that for every infinite $\RT^2_2(p)$-solution~$Y$ to~$g$, $f \oplus Y$ computes a $\coRT^2_2(p)$-solution to~$f$. Thus, $\RT^2_2(p)$ implies $\RT^2_2$ with only 2 applications.

\begin{definition}
Fix $k, n \geq 1$.
Let $f : [\NN]^2 \to 2$ be a coloring and $F$ be a set $f$-realizing the $k$-fractal of dimension~$n$.
The \emph{fractal decomposition} of~$F$ is the sequence $F_0 < \dots < F_{k-1}$ of subsets of~$F$
such that for every~$i < k$, $F_i$ is a $k$-fractal of dimension~$n-1$. We shall refer to the $F_i$'s as \emph{blocks}.
\end{definition}

Recall that given an infinite set~$X \subseteq \NN$, we say that an element~$x \in \NN$ has \emph{limit $c$ in~$X$} if for all but finitely many~$y \in X$, $f(x, y) = c$. The following definition generalizes the notion of~$X$-good/$X$-bad from \Cref{sec:sep-perm-measure}. Indeed, an $i$-fractal of dimension~1 is nothing but an $f$-homogeneous set for color~0 of size~$i$. With the following definition, a finite $f$-homogeneous set for color~$0$ is $(f, k)$-good for color~$0$ in~$X$ if it has at most $k-1$ $X$-bad elements, that is, at most $k-1$ elements which have limit~$1$ in~$X$.

\begin{definition}
Fix $k, i \geq 1$ and $n \in \NN$.
Let $f : [\NN]^2 \to 2$ be a coloring, $X$ be an infinite set and $c < 2$ be a color.
A finite set~$F \subseteq \NN$ $f$-realizing the $i$-fractal of dimension~$n$ is \emph{$(f,k)$-good for color~$c$ in~$X$} if
there is no subset~$G \subseteq F$ $f$-realizing the $k$-fractal of dimension~$n$ with all its elements having limit $1-c$ in~$X$.
Otherwise, $F$ is \emph{$(f,k)$-bad for color~$c$ in~$X$}.
\end{definition}

Note that in the case $n = 0$, $F$ is a singleton and its unique element does not have limit $1-c$ in~$X$.

The following lemma bounds the number of bad choices of fractals of dimension~$n-1$ given a good fractal of dimension~$n$. This yields a probabilistic algorithm to find an element which does not have limit~$1-c$ in~$X$ given an $i$-fractal~$F$ of dimension~$n$ which is $(f, k)$-good for color~$c$ in~$X$: pick a block~$F'$ at random in the fractal decomposition of~$F$. With probability $1-\frac{k-1}{i}$, $F'$ is an $i$-fractal of dimension~$n-1$ which is $(f, k)$-good for color~$c$ in~$X$. Iterate the process until we obtain an $i$-fractal of dimension~0 which is $(f,k)$-good for color~$c$ in~$X$, in other words, a singleton element which does not have limit~$1-c$ in~$X$. The probability of making a bad choice at some point is at most $\frac{(k-1) \times n}{i}$, which can be made arbitrarily small by considering $i$-fractals for very large~$i$'s.

\begin{lemma}\label[lemma]{lem:bad-blocks-bounded}
Fix $k, i, n \geq 1$.
Let $f : [\NN]^2 \to 2$ be a coloring, $X$ be an infinite set and $c < 2$ be a color.
Let $F \subseteq \NN$ be a set $f$-realizing the $i$-fractal of dimension~$n$ and let $F_0 < \dots < F_{i-1}$ be its fractal decomposition.
If $F$ is $(f, k)$-good for color~$c$ in~$X$, then there are at most $k-1$ many blocks which are $(f,k)$-bad for color~$c$ in~$X$.
\end{lemma}
\begin{proof}
Assuming by contrapositive that there exists indices $\ell_0 < \dots < \ell_{k-1} < i$ such that, for every $j < k$, the blocks $F_{\ell_j}$ is $(f,k)$-bad for color~$c$ in~$X$. Since $n \geq 1$ and by definition of a $(f,k)$-bad block, for every $j < k$ there exists a subset $G_{\ell_j} \subseteq F_{\ell_j}$ such that $G_{\ell_j}$ $f$-realizes the $k$-fractal of dimension $n-1$ and all its elements have limit $1-c$ in $X$.

Then, $G_{\ell_0} \cup \dots \cup G_{\ell_{k-1}}$ $f$-realizes the $k$-fractal of dimension $n$ and all its elements have limit $1-c$ in $X$, hence, $F$ is $(f, k)$-bad for color~$c$ in~$X$. 
\end{proof}

The proof of \Cref{lem:DNC-avoid-infinite} is based on a case analysis: either the $\emptyset'$-escaping function~$g$ computes $\emptyset''$, in which case we use the existence of a $\Delta^0_3$ solution, or $g$ has infinitely many relatively small values, and can be considered as a partial computably bounded $\emptyset'$-escaping function.

We now prove \Cref{prop:2dnc-computes-homogeneous}, which is the counterpart of \Cref{prop:2random-computes-homogeneous}, but using an elaboration of \Cref{lem:DNC-avoid-infinite} instead of the existence of a 2-random.
There is however an additional level of difficulty: the probabilistic algorithm given before \Cref{lem:bad-blocks-bounded} requires to make $n$ iterated queries to the $\emptyset'$-escaping function, asking each time to give an answer smaller than~$i$.

\begin{proposition}\label[proposition]{prop:2dnc-computes-homogeneous}
Fix a computable coloring $f : [\NN]^2 \to 2$, some color~$c < 2$ and some~$k, n \geq 1$.
Suppose that for every~$i \in \NN$ and every infinite computable set~$X$, the following holds:
\begin{itemize}
    \item[(1)] There is a subset $F \subseteq X$ $f$-realizing the $i$-fractal of dimension~$n$;
    \item[(2)] Every set~$F \subseteq X$ $f$-realizing the $i$-fractal of dimension~$n$ is $(f,k)$-good for color~$c$ in~$X$.
\end{itemize}
Then every $\emptyset'$-DNC degree $\dbf$ computes an infinite $f$-homogeneous set.
\end{proposition}

\begin{proof}
Fix $f$, $c$, $k$, $n$ and $\dbf$. 

Suppose $\dbf \not \geq_T \emptyset''$, otherwise, by Jockusch~\cite[Theorem 4.2]{jockusch1972ramsey}, $\dbf$ computes an infinite $f$-homogeneous set and we are done. Let $Z \leq_T \emptyset'$ be such that $\dbf \not \geq_T Z'$ and $\dbf \geq_T Z$.

Let $g \leq_T \dbf$ be the $\emptyset'$-escaping function of \Cref{lem:equiv-dnc-avoid}(2). Let $\mu_s$ be the standard left-c.e.\ approximation of the modulus $\mu$ of $Z'$, that is, $\mu_s(x)$ is the least~$t \leq s$ such that $Z'_t \uh x = Z'_s \uh x$. For every $x \in \NN$, the sequence $(\mu_s(x))_{s \in \NN}$ can only change values at most $x$ times (every time a program with Turing index less than $x$ stops), and let $m^x_0 < \dots < m^x_i (= \mu(x))$ be those values for some $i < x$ which depends on~$x$.

For every finite set $E \subseteq \NN$, let 
$$
X_E = \{y \in \NN : \forall x \in E, x < y \wedge f(x,y) = c\}
$$

Given an $i$-fractal $F$ of dimension~$n$, and $\sigma \in \NN^{\leq n}$, let $F(\sigma)$ be defined inductively as follows:
$F(\langle\rangle) = F$ and $F(w \cdot \tau) = F_w(\tau)$, where $F_0 < \dots < F_{i-1}$ is the fractal decomposition of~$F$, if $w < i$. If $w \geq i$, then $F(w \cdot \tau) = \emptyset$.
Note that if $F(\sigma) \neq \emptyset$, then $F(\sigma)$ is an $i$-fractal of dimension $n-|\sigma|$.

As $\emptyset' \geq Z$, can consider $e : \NN^3 \to \NN$ a computable function such that, for every~$\ell < n$, $(w_0, \dots, w_{\ell-1}) \in \NN^\ell$ and every finite set~$E \subseteq \NN$, $W^{\emptyset'}_{e(x, E, \langle w_0, \dots, w_{\ell-1} \rangle)}$ is the $\emptyset'$-c.e.\ set which, for every $j < x$ searches the smallest $F_j \subseteq X_E$ $f$-realizing the $m_j^x$-fractal of dimension $n$, and, if such an $F_j$ is found, outputs the indices of the bad blocks in the fractal decomposition of $F_j(\langle w_0, \dots, w_{\ell-1} \rangle)$.


For every~$\ell < n$, let $w_\ell$ be the $\dbf$-computable function defined inductively as follows:
$$
w_\ell(x, E) = \begin{cases}
    g(e(x, E, \langle\rangle), (k-1) \times x) & \mbox{ if } \ell = 0\\
    g(e(x, E, \langle w_0(x), \dots, w_{\ell-1}(x)\rangle), (k-1) \times x) & \mbox{ if } \ell > 0
\end{cases}
$$
\smallskip

\textbf{Claim 1:} For every finite set~$E \subseteq \NN$ such that $X_E$ is infinite, for every $x, j \in \NN$ such that  $\max \{ w_\ell(x, E) : \ell < n\}< m_j^x$ and every~$\ell \leq n$,
letting $F^j \subseteq X_E$ be the least $m^x_j$-fractal of dimension~$n$,
\begin{itemize}
    \item[(a)] $F^j(\langle w_0(x, E) \dots w_{\ell-1}(x, E)\rangle)$ is $(f,k)$-good for color~$c$ in~$X_E$;
    \item[(b)] $\card W^{\emptyset'}_{e(x, E, \langle w_0(x,E), \dots, w_{\ell-1}(x,E) \rangle)} \leq (k-1) \times x$
\end{itemize}
By induction on $\ell$. If $\ell = 0$, (a) is due to the fact that every fractal of dimension $n$ is $(f,k)$-good for color~$c$ in~$X_E$ and (b) follows from \Cref{lem:bad-blocks-bounded}. Then, assuming the property to be true for some $\ell < n$, the set $F^j(\langle w_0(x, E) \dots w_{\ell-1}(x, E)\rangle)$ is therefore an $m^x_j$-fractal of dimension~$n-\ell$ which is $(f, k)$-good for color~$c$ in~$X_E$. By \Cref{lem:bad-blocks-bounded}, there are at most $k-1$ blocks in the fractal decomposition of $F^j(\langle w_0(x, E) \dots w_{\ell-1}(x, E)\rangle)$ that are $(f,k)$-bad for color~$c$ in~$X_E$. The indices of these blocks are the one outputted by $W^{\emptyset'}_{e(x, E, \langle w_0(x,E), \dots, w_{\ell-1}(x,E) \rangle)}$ for every $j < x$ such that $F_j$ is found. Hence $\card W^{\emptyset'}_{e(x, E, \langle w_0(x,E), \dots, w_{\ell-1}(x,E) \rangle)} \leq (k-1) \times x$ and (b) holds. Thus, by definition of $g$, $w_{\ell}(x,E)$ does not belongs to $W^{\emptyset'}_{e(x, E, \langle w_0(x,E), \dots, w_{\ell-1}(x,E) \rangle)}$ and $F^j(\langle w_0(x, E), \dots, w_{\ell-1}(x, E), w_{\ell}(x,E)\rangle)$ is an $m^x_j$-fractal of dimension~$n-\ell-1$ which is $(f,k)$-good for color~$c$ in~$X_E$, so (a) holds.
This proves Claim~1.
\smallskip


We will define a $\dbf$-computable increasing sequence of computable Mathias conditions of the form $(E_0, X_{E_0}) \geq (E_1, X_{E_1}) \geq \dots$ such that for every~$s$, $\card E_s = s$.
Given $(E_s, X_{E_s})$, apply the following steps:
\begin{itemize}
    \item[(S1)] Search $\dbf$-computably for some~$x, j \in \NN$ such that 
$$\max \{ w_\ell(x, E_s) : \ell < n\}< m_j^x$$
Such $x, j$ is eventually found. Indeed, for infinitely many $x$ we have $\max \{ w_\ell(x, E_s) : \ell < n \} < \mu(x)$, otherwise $\dbf$ would compute $\emptyset'$, contradicting our assumption. 
    \item[(S2)] Once $x, j$ found, look for the smallest set $F \subseteq X_{E_s}$  $f$-realizing the $m_j^x$-fractal of dimension $n$. Such an $F$ will be found by (2) since $X_{E_s}$ is infinite.
    \item[(S3)] By Claim 1, $F_j(\langle w_0(x, E), \dots, w_{n-1}(x, E)\rangle)$ is $(f, k)$-good for color~$c$ in~$X_{E_s}$. Since $F_j(\langle w_0(x, E), \dots, w_{n-1}(x, E)\rangle)$ is an $i$-fractal of dimension~0, it is a singleton $\{x_0\}$. By definition of $(f, k)$-good for color~$c$ in~$X_{E_s}$, the set $X_{E_s \cup \{x_0\}}$ is infinite.
    Let $E_{s+1} = E_s \cup \{x_0\}$.
\end{itemize}
The set $H = \bigcup_s E_s$ is a $g$-computable $f$-homogeneous set for color~$c$.
\end{proof}

We can now prove our main proposition in the case of $k$-fractals.

\begin{proposition}\label[proposition]{prop:2dnc-fractal-avoiding}
    Let $f : [\NN]^2 \to 2$ be a computable coloring. For every $n,k \in \NN$, if $f$ avoids the $k$-fractal of dimension $n$, then every $\emptyset'$-DNC degree $\dbf$ computes an infinite $f$-homogeneous set.
\end{proposition}

\begin{proof}
The proof is very similar to that of \Cref{prop:probabilistic-fractal-avoiding}, except Case~1 which is treated differently.
Fix a computable coloring $f : [\NN]^2 \to 2$.  By induction on $n$, we will show that for any $k \in \NN$, any infinite computable set $D \subseteq \NN$ $f$-avoiding the $k$-fractal of dimensions $n$ and any $\emptyset'$-DNC degree $\dbf$, $\dbf$ computes an infinite $f$-homogeneous set.

The case $n = 0$ is vacuously true, as no set $f$-avoids the singleton pattern.

In the case $n = 1$, if some infinite computable set~$D$ $f$-avoids the $k$-fractal of dimension~1, then there is no $f$-homogeneous subset of~$D$ for color~0 of size~$k$. Then, by \Cref{lem:unbalanced-ramsey}, there exists an infinite computable $f$-homogeneous subset of~$D$ for the color $1$

Assume the property to be true for some $n \in \NN$. Fix an infinite computable set $D$ such that some $\emptyset'$-DNC degree $\dbf$ does not compute any infinite $f$-homogeneous set. By our induction hypothesis, for every $k' \in \NN$, and every infinite computable set $X \subseteq D$, there is some set $S \subseteq X$ $f$-realizing the $i$-fractal of dimension~$n$.

Fix a size $k \in \NN$ and let $c$ be equal to $1$ if $n$ is even and $0$ if $n$ is odd. Consider the two following cases:
\smallskip


\textbf{Case 1:} There exists some infinite computable set $G \subseteq D$, such that, for every infinite computable subset $H \subseteq G$, every $S \subseteq H$ $f$-realizing the $k$-fractal of dimension $n$ has at least one element which does not have limit $1-c$ in $H$.

Let $G = \{x_0 < x_1 < \dots \}$ and let $g : [\NN]^2 \to 2$ be the computable coloring defined by $g(i,j) = f(x_i,x_j)$ for every $i < j$. Fix some infinite computable subset $H \subseteq \NN$ and some $i \in \NN$. By our hypothesis on $G$, every $S \subseteq \NN$ $g$-realizing the $i$-fractal of dimension $n$ will be $(g,k)$-good for color $c$ in $H$ 
and by the induction hypothesis, there will be some subset $S \subseteq H$ $g$-realizing the $i$-fractal of dimension $n$.

Hence, by \Cref{prop:2dnc-computes-homogeneous}, $\dbf$ computes an infinite $f$-homogeneous set, contradicting our assumption. Hence, this case is impossible.

\smallskip

\textbf{Case 2:} For every infinite computable set $G \subseteq D$, there exists an infinite computable set $H \subseteq G$ and some $S \subseteq H$ $f$-realizing the $k$-fractal of dimension $n$ with all its element having limit $1 - c$ in $H$. This case is treated exactly as in the proof of \Cref{prop:probabilistic-fractal-avoiding}.

\end{proof}

\begin{theorem}\label[theorem]{thm:cort22-sep-dnczp}
For every separable permutation~$p$ and every computable instance of~$\coRT^2_2(p)$,
every $\emptyset'$-DNC degree computes a solution.
\end{theorem}
\begin{proof}
Let $f : [\NN]^2 \to 2$ be a computable coloring avoiding~$p$.
By \Cref{prop:fractal-basis-separable}, there exists some $k,n \in \NN$ such that it avoids the $k$-fractal of dimension $n$. Then by \Cref{prop:2dnc-fractal-avoiding} every $\emptyset'$-DNC degree computes an infinite $f$-homogeneous set.
\end{proof}

As mentioned, Mimouni and Patey~\cite{mimouni2025ramseylike} proved that $\RT^2_2(p)$ for any non-constant pattern~$p$ is not computably true by constructing a computable coloring $f : [\NN]^2 \to 2$ such that for every infinite set~$H$ $f$-avoiding any pattern, $H$ computes a $\emptyset'$-DNC function. Combined with \Cref{thm:cort22-sep-dnczp}, this yields the following corollary:

\begin{corollary}\label[corollary]{cor:seprt22-rt22}
$\SepRT^2_2$ implies $\RT^2_2$ over $\omega$-models.
\end{corollary}
\begin{proof}
Let $\M$ be an $\omega$-model of~$\SepRT^2_2$
and let $f : [\NN]^2 \to 2$ be a coloring in~$\M$.
By $\SepRT^2_2$, there is an infinite set~$X \in \M$ which $f$-avoids some separable permutation~$p$.
By Mimouni and Patey~\cite{mimouni2025ramseylike}, there is an $X'$-DNC function~$g$ in~$\M$.
By a relativized version of \Cref{thm:cort22-sep-dnczp}, $g \oplus X$ computes an infinite $f$-homogeneous set~$H$.
In particular, $H \in \M$ since~$\M$ is a Turing ideal.
\end{proof}

We are now ready to prove our main theorem.

\begin{repmaintheorem}{maintheorem1}
For every pattern~$p$, $\RT^2_2(p)$ implies $\RT^2_2$ over $\omega$-models if and only if $p$ is a separable permutation.
\end{repmaintheorem}
\begin{proof}
If $p$ is a separable permutation, then $\RCA_0 \vdash \RT^2_2(p) \to \SepRT^2_2$, so by \Cref{cor:seprt22-rt22}, $\RT^2_2(p)$ implies $\RT^2_2$ over $\omega$-models.
If $p$ is not a separable permutation, then by \Cref{cor:sep-permutation-not-rt22}, $\RT^2_2(p)$ does not imply $\RT^2_2$ over~$\omega$-models.
\end{proof}

Note that the bound of \Cref{thm:cort22-sep-dnczp} is not optimal as for every separable permutation~$p$, every computable instance of~$\coRT^2_2(p)$ admits an $\emptyset'$-computable solution, while there is no such $\emptyset'$-DNC function. Moreover, for the case $p = 120$, Cervelle, Gaudelier and Patey~\cite{cervelle2024reverse} proved that $\ADS$ implies $\coRT^2_2(120)$ over~$\RCA_0$, and by Hirschfeldt and Shore~\cite{hirschfeldt2007combinatorial}, $\ADS$ does not even imply the existence of a DNC function over $\omega$-models.

\section{Linear orders and measure}\label[section]{sec:linear-order-measure}

In \Cref{sec:sep-perm-measure,sec:2dnc-sep}, we proved that $\coRT^2_2(p)$ admits probabilistic solutions when $p$ is a separable permutation, and gave computability-theoretic upper bounds in terms of relativized diagonal non-computability. In this section, we prove the tightness of \Cref{thm:avoid-separable-randomness} by proving that if $p$ is a pattern which is not a separable permutation, then $\coRT^2_2(p)$ admits no probabilistic solutions, in the sense that there is a computable instance such that the measure of oracles computing a solution is~0. For this, we use some interactions between computable linear orders and measure.


Given a linear order $\L = (\NN, <_\L)$ of order type $\omega+\omega^*$.
We write $A(\L)$ for the $\omega$-part of $\L$, that is, the set $\{ x \in \NN : \forall^\infty y\ x <_\L y \}$.
A sequence is \emph{$\L$-monotonous} if it is $\L$-ascending or $\L$-descending.

The following proposition is a consequence of \cite[Theorem 4.1,Theorem 4.2]{jockusch1968semirecursive}. Indeed, for every computable linear order $\L = (\NN, <_\L)$ of order type $\omega+\omega^*$, by \cite[Theorem 4.1]{jockusch1968semirecursive} $A(\L)$ and $\overline{A}(\L)$ are semirecursive. If furthermore $\L$ has no infinite computable $\L$-monotonous sequence, then $A(\L)$ and  $\overline{A}(\L)$ are both immune, hence by \cite[Theorem 4.2]{jockusch1968semirecursive} they are both hyperimmune.
We give a more direct proof without using the notion of semirecursive set.

\begin{proposition}\label[proposition]{prop:sads-bihyperimmune}
Let $\L = (\NN, <_\L)$ be a computable linear order of order type $\omega+\omega^*$, with no computable infinite $\L$-monotonous sequence. Then $A(\L)$ is bi-hyperimmune.
\end{proposition}
\begin{proof}
Suppose for the contradiction that $A(\L)$ is not hyperimmune.
Let $(F_n)_{n \in \NN}$ be a c.e.\ array such that $F_n \cap \A(\L) \neq \emptyset$ for every~$n \in \NN$.
Then, the set $H = \{ \min_\L(F_n) : n \in \NN \}$ is an infinite c.e.\ subset of~$A(\L)$. The set~$H$ contains  a computable infinite $\L$-ascending sequence, contradicting the hypothesis.
By a dual argument, $\overline{A}(\L)$ is hyperimmune since there is no computable infinite $\L$-descending sequence.
\end{proof}

Mileti~\cite[Theorem 5.2.6]{mileti2004partition} proved that for every infinite hyperimmune set~$A \subseteq \NN$, the measure of oracles computing an infinite subset of~$A$ is~0. It follows that if $\L = (\NN, <_\L)$ is a computable instance of~$\SADS$ with no computable solution, the measure of oracles computing a solution is~0. We prove it directly for the sake of containment and simplicity.

\begin{proposition}\label[proposition]{prop:sads-measure-0}
Let $\L = (\NN, <_\L)$ be a computable linear order of order type $\omega+\omega^*$, with no computable infinite $\L$-monotonous sequence. The measure of oracles computing an infinite $\L$-monotonous sequence is~0.  
\end{proposition}
\begin{proof}
Assume by contradiction that this measure is non-zero. There are countably many functionals $(\Phi_e)_{e \in \NN}$, hence there exists one such functional $\Phi$ such that the class of oracles $A \subseteq \NN$ for which $\Phi^A$ is an infinite $\L$-monotonous sequence is also of non-zero measure (as a countable union of measure zero sets is still of measure zero). Without any loss of generality, assume that $\Phi$ only outputs $\L$-ascending sequences for every oracle.

By Lebesgue's density lemma, there exists a finite binary sequence $\sigma$ such that the measure of oracles $A$ in the open class $[\sigma] := \{A : \sigma \prec A\}$ such that $\Phi^A$ is an infinite $\L$-ascending sequence is greater than half the measure of $[\sigma]$. By considering a functional $\Psi$ such that $\Psi^A = \Phi^{\sigma\cdot A}$, we obtain a functional outputting an infinite $\L$-ascending sequence for a set of oracles of measure greater than~$0.5$.

For every $x \in \A(\L)$, there are only finitely many elements $<_{\L}$-smaller than $x$, hence every infinite $\L$-ascending sequence (in particular those outputted by $\Psi$) is eventually bigger than $x$. Therefore, for every such $x$, there exists a finite set $W_x$ of finite binary sequences such that $[W_x] := \bigcup_{\sigma \in W_x} [\sigma]$ is of measure greater than $0.5$ and such that for every $\sigma \in W_x$, $\Psi^{\sigma}$'s output is a finite $\L$-ascending sequence whose last element is $<_{\L}$-bigger than $x$. At least one of those $\sigma$ is a prefix of a set $A$ such that $\Psi^A$ is an infinite $\L$-ascending sequence, as we cannot have two disjoint sets of measure greater than $0.5$ in a set of measure $1$, hence, for this $\sigma$, the last element outputted by $\Psi^{\sigma}$ will also be in $\A(\L)$. Therefore, by considering the last elements of the sequences $\Psi^{\sigma}$ for $\sigma \in W_x$ and by picking the $<_{\L}$-smallest of them, we obtained an element $y \in \A(\L)$ such that $x <_{\L} y$. Such sets $W_x$ and elements $y$ can be computably uniformly found for every $x \in \A(\L)$. Thus, using $\Psi$ we can find a computable infinite $\L$-ascending sequence, contradicting our assumptions on $\L$.
\end{proof}

We shall see in \Cref{thm:linear-probabilistic-noncomputable} that the previous proposition does not hold for arbitrary computable linear orderings.
One can derive many consequences from \Cref{prop:sads-bihyperimmune} and \Cref{prop:sads-measure-0}.
First of all, we reprove \Cref{prop:separable-trivial-ads} in the particular case of linear orders of order type $\omega + \omega^*$.

\begin{corollary}\label[corollary]{cor:linear-separable-trivial}
Let $\L = (\NN, <_\L)$ be a computable linear order of order type~$\omega+\omega^*$ avoiding any separable pattern.
Then there is an infinite computable $\L$-monotonous sequence.
\end{corollary}
\begin{proof}
By \Cref{thm:avoid-separable-randomness}, the measure of oracles computing an infinite $\L$-monoto\-nous sequence is~1. Thus, by \Cref{prop:sads-measure-0}, there is a computable infinite such sequence.
\end{proof}



The following proposition yields an example of an instance of~$\SADS$ avoiding every pattern, except the separable permutations. Indeed, since it is a linear order, the only patterns appearing are permutations, and by \Cref{thm:characterization-sep-bose}, if a permutation is not separable, it contains either 1302 or 2031 as sub-pattern, and therefore does not appear in the linear ordering.

Furthermore, since this instance of~$\SADS$ does not contain any solution computable in the instance, by \Cref{cor:linear-separable-trivial}, every separable pattern appears in it. This linear order therefore has the property that no further pattern can be avoided without computing a solution.

\begin{proposition}\label[proposition]{prop:linear-order-avoid-all-but-sep}
There exists a linear order $\L = (\NN, <_\L)$ of order type $\omega+\omega^{*}$ avoiding the permutations 1302 and 2031, with no $\L$-computable infinite $\L$-monotonous sequence.
\end{proposition}
\begin{proof}
Let $\L_0 = (\NN, <_{\L_0})$ be any computable linear order of order type $\omega+\omega^{*}$ with no computable infinite $\L_0$-monotonous sequence. By \Cref{prop:sads-bihyperimmune}, $\A(\L_0)$ is bi-hyperimmune.
Since 1302 and 2031 are non-separable permutations, by \Cref{thm:preserve-hyps-irr-div}, $\RT^2_2(1302)$ and $\RT^2_2(2031)$ both admit preservation of $\omega$ hyperimmunities. Therefore, there is an infinite set $H = \{ x_0 < x_1 < \dots \} \subseteq \NN$ which $\L_0$-avoids 1302 and 2031 and such that $\A(\L_0)$ is bi-$H$-hyperimmune. In particular, there is no infinite $H$-computable $\L_0$-monotonous sequence.
Let $\L = (\NN, <_\L)$ be the $H$-computable linear order defined by $n <_\L m$ iff $x_n <_{\L_0} x_m$.
Then $\L$ has no $\L$-computable monotonous sequence, and avoids the permutations 1302 and 2031.
\end{proof}

The following corollary shows the tightness of \Cref{sec:sep-perm-measure}, in that the only statements of the form $\coRT^2_2(p)$ admitting probabilistic solutions are for separable permutations~$p$.

\begin{corollary}
For every pattern~$p$ which is not a separable permutation, there is an instance $f : [\NN]^2 \to 2$ of~$\coRT^2_2(p)$ such that the measure of oracles $f$-computing an infinite $f$-homogeneous set is~0.
\end{corollary}
\begin{proof}
Let $\L = (\NN, <_\L)$ be the linear order of \Cref{prop:linear-order-avoid-all-but-sep}.
Seeing $\L$ as a transitive stable coloring $f : [\NN]^2 \to 2$, it $f$-avoids 1302 and 2031,
so by \Cref{thm:characterization-sep-bose}, it $f$-avoids~$p$. Moreover, by a relativization of \Cref{prop:sads-measure-0}, the measure of oracles $f$-computing an infinite $f$-homogeneous set is~0.
\end{proof}

One can push further the argument, and show that any computable instance of~$\SADS$ with no computable solution is a witness that $\RT^2_2(p)$ does not imply $\SADS$ on $\omega$-models, for any $\RT^2_2$-pattern~$p$ other than the separable permutations. Since $\SepRT^2_2(p)$ implies $\ADS$ on~$\omega$-models, this result is tight.

\begin{proposition}
Let $\L = (\NN, <_\L)$ be a computable linear order of order type $\omega + \omega^*$, with no infinite computable $\L$-monotonous sequence. Then there is an $\omega$-model $\M$ of $\RCA_0 + \EM + \RT^2_2(1302)$ such that $\M$ does not contain any infinite $\L$-monotonous sequence.
\end{proposition}
\begin{proof}
Fix~$\L$. We construct an $\omega$-model~$\M$ of $\EM + \RT^2_2(1302)$ such that for every~$X \in \M$, $\A(\L)$ is bi-$X$-hyperimmune.  More precisely, we construct an infinite sequence $Z_0 \leq_T Z_1 \leq_T \dots$ of sets such that for every~$n \in \NN$,
\begin{itemize}
    \item[(1)] If~$n = \langle a, b\rangle$ and $\Phi^{Z_a}_b$ is a coloring $g : [\NN]^2 \to 2$, then $Z_{n+1}$ computes an infinite $g$-transitive set $g$-avoiding~$1302$;
    \item[(2)] $\A(\L)$ and $\overline{\A}(\L)$ are both~$Z_n$-hyperimmune.
\end{itemize}
First, let $Z_0 = \emptyset$. Assuming $Z_n$ is defined, letting $n = \langle a, b \rangle$, if $\Phi^{Z_a}_b$ is not a coloring, then let $Z_{n+1} = Z_n$. Otherwise, since $\EM$ and $\RT^2_2(1302)$ both admit preservation of $\omega$ hyperimmunities, there is an infinite $\Phi^{Z_a}_b$-transitive set~$H$ $\Phi^{Z_a}_b$-avoiding~$1302$ such that $\A(\L)$ and $\overline{\A}(\L)$ are both $H \oplus Z_n$-hyperimmune. Let $Z_{n+1} = H \oplus Z_n$. This completes our construction.
Let $\M$ be the $\omega$-model whose second-order part is $\{ X : \exists n\ X \leq_T Z_n \}$. By (1), $\M \models \EM + \RT^2_2(1302)$, and by (2), there is no infinite $\L$-monotonous sequence in~$\M$. 
\end{proof}


Mimouni and Patey~\cite{mimouni2025ramseylike} proved some general lower bounds on the complexity of avoiding any pattern. In particular, they constructed a computable coloring $f : [\NN]^2 \to 2$ such that every infinite set $f$-avoiding any pattern computes a $\emptyset'$-DNC function, and another computable coloring $g : [\NN]^2 \to 2$ such that the measure of oracles computing a solution is~0. One cannot refine this first lower bound by constructing a computable linear order $\L = (\NN, <_\L)$ such that every infinite set $f$-avoiding any permutation computes a $\emptyset'$-DNC function, or even a DNC function, as any infinite $\L$-monotonous sequence avoids all the non-constant patterns, and $\ADS$ does not  imply the existence of a DNC function (see Hirschfeldt and Shore~\cite{hirschfeldt2007combinatorial}). We however refine the second lower bound by constructing a computable linear order.

The following proof is a straightforward adaptation of \cite[Theorem 2.2]{mimouni2025ramseylike} for the case of a transitive coloring. Indeed, when running the construction of \cite[Theorem 2.2]{mimouni2025ramseylike} only for transitive patterns, the resulting coloring is itself transitive, hence codes a linear order. We include the whole proof for the sake of completeness.

\begin{theorem}\label[theorem]{thm:measure-avoid-pattern}
There is a computable linear ordering $\L = (\NN, <_\L)$ of order type $\omega + \omega^*$ such that the measure of oracles computing an infinite set $\L$-avoiding any transitive pattern is~0.
\end{theorem}

\begin{proof}
The ordering $\L$ will be the one given by a stable transitive coloring $f : [\NN]^2 \to 2$ built using a finite-injury priority construction: the value of $f(x,s)$ will be determined at stage~$s$ of the construction for every $x < s$. At any stage $s$ of the construction, every element $x$ of $\NN$ will commit to have a certain limit behavior $c < 2$, and we will set $f(x,s)$ accordingly. Initially, every element commits to have limit behavior $0$.

The coloring $f$ will satisfy the following requirements for every transitive pattern~$p$ and every Turing index~$e$:
\begin{quote}
    $\R_{p,e}$: $\mu(\{ X \in \cs : W_e^X \mbox{ is finite or } f\mbox{-contains } p \}) \geq \frac{1}{2|p|}$.
\end{quote}
We first claim that if all $\R$-requirements are satisfied, then $\L$ satisfies the statement of the theorem, using the contrapositive.
Suppose that the measure of oracles computing an infinite set avoiding any transitive pattern for~$f$ is positive.
Then, since there are countably many transitive patterns and countably many functionals, there is some transitive pattern~$p$ and some Turing index~$e$ such that the measure of oracles~$X$ such that $\Phi_e^X$ is infinite and $f$-avoids~$p$ is positive. Indeed, a countable union of classes of measure~0 is again of measure~0. By the Lebesgue density theorem, there is some string $\sigma \in \bstr$ such that  the measure of oracles $X$ such that $\Phi_e^{\sigma \cdot X}$ is infinite and $f$-avoids~$p$ is more than $1-\frac{1}{2|p|}$.
Let~$a$ be a Turing index such that $\Phi_a^X = \Phi_e^{\sigma \cdot X}$. Then the requirement $\R_{p, a}$ is not satisfied. 

The strategies are given a priority order based on Cantor's pairing function $\langle p, e\rangle$, the smaller value being of higher priority. Each requirement $\R_{p, e}$ will be given a movable marker~$m_{p,e}$, starting at~$\langle p,e\rangle$, such that if $\R_{q, i}$ is of greater priority than $\R_{p, e}$, then $m_{q, i} \leq m_{p, e}$. 
\smallskip

\textbf{State of $\R_{p,e}$.}
Each strategy $\R_{p, e}$ is given a \emph{state}, which is a finite sequence $\langle F_0, \dots, F_{t-1}\rangle$ of non-empty finite sets, with $t \leq |p|$ and such that $F_i < F_{i+1}$. Such a sequence satisfies the following properties:
\begin{enumerate}
    \item[(P1)] For every~$x_0 \in F_0, \dots, x_{t-1} \in F_{t-1}$, $\{x_0, \dots, x_{t-1}\}$ $f$-realizes $p \uh_t$
    \item[(P2)] $\mu(\{ X : W_e^X \cap F_i \neq \emptyset \}) > 1-\frac{1}{2|p|}$
    \item[(P3)] Every $F_i$ is an $\leq_{\NN}$-interval such that $F_0 <\dots < F_{t-1}$ and there are no elements $x \in \NN$ such that $F_i < x < F_{i+1}$ for some $i < t- 1$
    \item[(P4)] If $t \neq 0$, then, for every other requirement $\R_{q,i}$, we have $m_{q,i} < F_0$ if $\langle q,i \rangle < \langle p, e \rangle$ and $m_{q,i} > F_{t-1}$ if $\langle q,i \rangle > \langle p, e \rangle$.
\end{enumerate}
Over time, new sets will be stacked to this list, which will be reset only if a strategy of higher priority injures it. Initially, each strategy is given the empty sequence as state.

\smallskip

\textbf{Strategy for $\R_{p,e}$.}
A requirement $\R_{p, e}$ \emph{requires attention at stage~$s$} if its state has length less than~$|p|$ and there is a finite set of initial segments of oracles $U \subseteq 2^{\leq s}$ such that $\sum_{\sigma \in U} 2^{-|\sigma|} > 1-\frac{1}{2|p|}$ and for every~$\sigma \in U$, $W_e^\sigma[s] \cap [m_{p,e}, s] \neq \emptyset$. In other words, $\R_{p,e}$ requires attention at stage~$s$ if the measure of oracles $X$ such that $W_e^X[s]$ outputs an element in $[m_{p,e}, s]$ is greater than $1-\frac{1}{2|p|}$.

If $\R_{p,e}$ receives attention at stage~$s$ and is in state~$\langle F_0, \dots, F_{t-1} \rangle$ (by convention, if the state is the empty sequence, $t = 0$), then, letting $F_t = [m_{p,e}, s]$, its new state is $\langle F_0, \dots, F_t \rangle$. The marker $m_{p,e}$ is moved to $s+1$, and the markers of all strategies of lower priorities is moved further, accordingly. All the strategies of lower priorities are injured, and their state is reset to the empty sequence. If $t < |p|-1$, then for every~$i \leq t$, all the elements of $F_i$ commit to have limit $p(i, t+1)$ and if $t = |p| - 1$, then all the elements of $F_i$ commit to have limit $0$.
\smallskip

\textbf{Construction.}
The global construction goes by stages, as follows. Initially, $f$ is nowhere-defined.
At stage~$s$, suppose $f$ is defined on $[0, s-1]^2$. If some strategy requires attention at stage~$s$, letting $\R_{p, e}$ be the strategy of highest priority among these, give it attention and act accordingly. 
In any case, for every~$x < s$, if $x$ is committed to have some limit~$c < 2$, then set $f(x, s) = c$.
\smallskip

\textbf{Verification.}
By construction, the state of every strategy will satisfy the properties (P1) to (P4).

The resulting coloring is transitive, indeed, suppose by contradiction that there exists $x < y < z$ and some $c < 2$ such that $f(x,y) = f(y,z) = c$ and $f(x,z) = 1 - c$. This means that at stage $y$ of the construction, $x$ was committed to have limit behavior $c$ and then changed his commitment between stage $y$ and $z$. Let $\R_{p,e}$ be the last strategy to have fixed the commitment of $x$ before stage $z$ (let say at some stage $t \in (y,z]$), hence all the strategies $\R_{q,i}$ of lower priority cannot change the commitment of $y$ between stage $t$ and $z$ (they have been injured and all of their markers have been pushed after $t$), and all the strategies of higher priority cannot change the commitment of $y$ either (by (P4), all of their marker are before $x$, hence they would also act on the commitment of $x$, contradicting our minimality assumption on $\R_{p,e}$). Thus, only $\R_{p,e}$ can change the commitment of $y$ between stages $t$ and $z$.

Let $\langle F_0, \dots, F_{\ell} \rangle$ be the state of $\R_{p,e}$ at stage $z$. By (P3), there exists some $i \leq j \leq \ell$ such that $x \in F_i$ and $y \in F_j$. We cannot have $i = j$, otherwise we would have $f(x,z) = f(y,z)$, contradicting our assumption, hence $i < j$. We cannot have $\ell = |p| - 1$ otherwise both $x$ and $y$ would have committed to have limit $0$, making $f(x,z) = f(y,z)$, again contradicting our assumption. This means that at stage $y$, $\R_{p,e}$ was in state $\langle F_0, \dots, F_{j-1} \rangle$, leading to $f(x,y) = p(i,j)$, but we would then have $f(x,z) = p(i,\ell)$ and $f(y,z) = p(j,\ell)$, and since $p$ is a transitive pattern, this also contradicts our assumption. \\





One easily sees by induction on the strategies that every strategy acts finitely often, and therefore each strategy is finitely injured by a strategy of higher priority. It follows that each requirement has a limit state and that $m_{p,e}$ reaches a limit value.

We claim that each requirement $\R_{p,e}$ is satisfied. Let $s$ be stage after which $m_{p, e}$ reaches its limit value. In particular, the state of~$\R_{p, e}$ also reached its limit, and none of the strategies of higher priority require attention after~$s$. Suppose first that the limit state of~$\R_{p, e}$ has length less than~$|p|$. This means that $\R_{p, e}$ does not require attention attention after stage~$s$, so the measure of oracles $X$ such that $W_e^X$ outputs an element greater than or equal to~$m_{p,e}$ is at most $1-\frac{1}{2|p|}$. Thus, $\mu(\{ X \in \cs : W_e^X \mbox{ is finite } \}) \geq \frac{1}{2|p|}$, and the requirement is therefore satisfied. Suppose now that the limit state of~$\R_{p,e}$ has length~$|p|$. By (P2), for each~$i < |p|$, $\mu(\{ X : W_e^X \cap F_i \neq \emptyset \}) > 1-\frac{1}{2|p|}$. It follows that 
$$
\mu(\{ X : (\forall i < |p|) W_e^X \cap F_i \neq \emptyset \}) > 1-\frac{|p|}{2|p|} = 1/2
$$
Thus, by (P2), $\mu(\{ X \in \cs : W_e^X f\mbox{-realizes } p \}) \geq \frac{1}{2}$, so $\R_{p,e}$ is again satisfied. \\

For every $x \in \NN$, there are finitely many markers starting at a position smaller than $x$, hence, once all their requirements have reach their limit state, the commitment of $x$ will not change anymore, ensuring that $x$ has a limit, and that the coloring is stable. The ordering corresponding to a stable transitive coloring is a linear order of order type of the form $\omega + k$, or $k + \omega^*$ or $\omega + \omega^*$ for some $k \in \NN$. The two first cases being impossible, otherwise there would be some infinite computable $f$-homogeneous set, the resulting ordering is of order type $\omega + \omega^*$.

This completes the proof of \Cref{thm:measure-avoid-pattern}.
\end{proof}

We conclude this section by proving that \Cref{prop:sads-measure-0} does not extend to arbitrary computable linear orderings.

\begin{theorem}\label[theorem]{thm:linear-probabilistic-noncomputable}
    There exists a computable linear ordering $(\NN, \sqsubseteq)$ with no computable infinite monotone sequence, such that the measure of oracle computing an infinite increasing sequence is~1.
\end{theorem}

\begin{proof}
It is sufficient to find a computable linear order $(A, \sqsubseteq_A)$ for $A$ an infinite set, as well as a functional that outputs an infinite $\sqsubseteq_A$-increasing sequence on a set of measure $> \frac{1}{4}$.
For that, we will consider four functionals $\Gamma_{inc}$, $\Gamma_{dec}$, $\Delta_{inc}$ and $\Delta_{dec}$ with the following properties:
\begin{itemize}
    \item For $A = \{x_0, x_1, \dots \}$ a set and $e$ a Turing index, $\Gamma^A_{inc}(e)$ and $\Gamma^A_{dec}(e)$ will compute a linear order $(B,\sqsubseteq_B)$ for some $B \subseteq A$. If $A$ is infinite, then so will be $B$.
    
    \item For every sets $A$ and $C$, $\Delta^{C\oplus A}_{inc}(e)$ (resp $\Delta^{C\oplus A}_{dec}(e)$) will compute an increasing sequence of $\Gamma^A_{inc}(e)$ (resp $\Gamma^A_{dec}(e)$). This sequence will also be increasing for the standard ordering $<$ of the integers.
 
    If $A$ is infinite, then the measure of oracles $C$ such that $\Delta^{C\oplus A}_{inc}(e)$ (resp $\Delta^{C\oplus A}_{dec}(e)$) is an infinite increasing sequence will be at least $\Pi_{\ell > e} (1 - \frac{1}{2^{\ell}})$. Note that $\Pi_{\ell > 0} (1 - \frac{1}{2^{\ell}}) > \frac{1}{4}$.
    
    \item No $W_{e'} \cap B$ for $e' > e$ will be an infinite monotone sequence of $\Gamma^A_{inc}(e)$ or $\Gamma^A_{dec}(e)$. Furthermore, $W_{e} \cap B$ will not be an infinite monotone
    sequence of $\Gamma^A_{dec}(e)$ and will not be an infinite decreasing sequence of $\Gamma^A_{inc}(e)$.
\end{itemize}

These four functionals will construct their ordering step by step, with the position of $x_s$ in the ordering (or the absence of $x_s$ in the ordering) being determined at stage $s$. $\Gamma_{inc}$ and $\Gamma_{dec}$ will call each other recursively as follows: For every infinite oracle~$A \subseteq \NN$, $\Gamma^A_{dec}(0)$ will call $\Gamma^A_{inc}(0)$, which will call $\Gamma^B_{dec}(1)$ for multiple~$B \subseteq A$, calling $\Gamma^B_{inc}(1)$, and so on... \\

Let $A = \{x_0, x_1, \dots \}$ be a set and $e$ be a Turing index. \\

\textbf{Definition of $\Gamma^A_{dec}(e)$.}
$\Gamma^A_{dec}(e)$ computes the order $(B, \sqsubseteq_B)$ defined as follows: $\Gamma^A_{dec}(e)$ is the same ordering as $\Gamma^A_{inc}(e)$, but, if at some stage $s$ we have $x_i \in W_e[s] \cap B$ for some $i < s$, then, all the $x_j$ for $j \geq s$ will be bigger than all the $x_i$ for $i < s$. The relative position of the $\{x_0, \dots, x_{s-1}\}$ and of the $\{x_{s}, \dots \}$ remain the one given by $\Gamma^{A}_{inc}(e)$.
This construction ensures that $W_e \cap B$ is not an infinite decreasing sequence of $\Gamma^A_{dec}(e)$, indeed, with this construction, there are only finitely many elements smaller than $x_i$ in the ordering $\Gamma^A_{dec}(e)$. \\

\textbf{Definition of $\Delta^{C \oplus A}_{dec}(e)$.}
Notice that any sequence that is both increasing for $\Gamma_{inc}^A(e)$ and for $<$ will also be an increasing sequence of $\Gamma_{dec}^A(e)$. Hence $\Delta^{C \oplus A}_{dec}(e)$ can be chosen to be the same as $\Delta^{C \oplus A}_{inc}(e)$ \\

\textbf{Definition of $\Gamma^A_{inc}(e)$.}
$\Gamma^A_{inc}(e)$ computes the order $(B, \sqsubseteq_B)$ defined as follows: at the first $2^{e+1}$ stages of the construction, we let $\{x_0, \dots, x_{2^{e+1}-1}\}$ be in $B$ with $x_0 \sqsubseteq_B \dots \sqsubseteq_B x_{2^{e+1} - 1}$, then, we partition the remaining elements $A \setminus \{x_0, \dots, x_{2^{e+1} - 1}\}$ into $2^{e+1}$ blocks $A_i = \{x_j \in A : j \geq 2^{e+1} \wedge j \equiv i \mbox{ mod } 2^{e+1}\}$ and let $(B_i, \sqsubseteq_{B_i})$ be the order computed by $\Gamma^{A_i}_{dec}(e+1)$ for every $i < 2^{e+1}$. The set $B$ will be a subset of $\{x_0, \dots, x_{2^{e+1} - 1}\} \cup \bigcup_{i < 2^{e+1}} B_i$, and the order $\sqsubseteq_B$ will coincide with the $\sqsubseteq_{B_i}$ on $B_i \cap B$ and will be such that $x_0 \sqsubseteq_B B_0 \cap B \sqsubseteq_B x_1 \sqsubseteq_B B_1 \cap B \sqsubseteq_B x_2 \sqsubseteq_B \dots \sqsubseteq_B B_{2^{e+1} - 1} \cap B$. To find out which elements of $\bigcup_{i < 2^{e+1}} B_i$ are in $B$, we proceed as follows: at every stage $s \geq 2^{e+1}$ of the construction, a single block $B_i$ will be disabled, which ensures that $x_s$ will be added to $B$ only if $x_s \not\in B_i$. Initially at stage $s_0 = 2^{e+1}$, $B_0$ is disabled, and if at some stage $s_1 \geq s_0$ an element $y_1 \in W_e[s_1]$ was already put in $B \cap B_{i_1}$ at a previous stage for some $0 < i_1 < 2^{e+1}$, then, we re-enable $B_0$ and disable $B_{i_1}$ until we found some stage $s_2$ and some element $y_2 \in W_e[s_2]$ that was already in $B \cap B_{i_2}$ for some $i_2 > i_1$, then we re-enable $B_{i_1}$ and disable $B_{i_2}$, and so on. 

This construction is well-defined. Indeed, to compute the position (or absence) of an element $x_s$ in the ordering, a finite amount of recursive calls are needed. This is due to the fact that $\Gamma^A_{inc}(e)$ immediately fixes the position of $x_0, \dots, x_{2^{e+1} - 1}$ without any recursive call, then, the position of the next $2^{e+2}$ elements only need the recursive call to $\Gamma^{A_i}_{dec}(e+1)$ which then call $\Gamma^{A_i}_{inc}(e+1)$, and so on for the next elements. 

This definition of $\Gamma^A_{inc}(e)$ ensures that $W_e \cap B$ will not be an infinite increasing sequence. Indeed, otherwise, there would be some $i < 2^{e+1}$ such that after a certain stage $s$, all the elements of $W_e \cap B$ are in $B_i$ (and since the sequence is increasing, no elements of $W_e \cap B$ are in $B_j$ for $j > i$), this would disable $B_i$ with no reactivation for the rest of the construction, ensuring that $B_i \cap B$ is finite and contradicting our assumption. \\

\textbf{Definition of $\Delta^{C \oplus A}_{inc}(e)$.}
Notice that in the construction of the order $\Gamma^A_{inc}(e)$, exactly one of the blocks $B_i$ will end up disabled, thus, for all the other blocks $B_j$, an infinite increasing sequence of $(B_j, \sqsubseteq_{B_j})$ will output an infinite increasing sequence of $(B, \sqsubseteq_B)$ by removing the finite amount of elements that are not in $B_j \cap B$. Hence, $\Delta^{C \oplus A}_{inc}(e)$ can be defined to act as follows: using the first $e+1$ bits of $C$ it will pick one $j < 2^{e+1}$, then $x_j$ will be the smallest element of the increasing sequence outputted, and it will then recursively call $\Delta^{(C - (e + 1)) \oplus A_j}_{dec}(e+1)$ and remove the elements of the sequence that are not in $B_j \cap B$ to find the rest of the sequence.

$\Delta^{C \oplus A}_{inc}(e)$ will output an infinite increasing sequence if it never picks the wrong block that end up disabled during the recursive calls. This happens with probability $\Pi_{\ell > e} (1 - \frac{1}{2^{\ell}})$. \\

So far, we only proved that $\Gamma^{A}_{dec}(e)$ (resp $\Gamma^{A}_{inc}(e)$) ensures that $W_e \cap B$ will not be an infinite decreasing (resp increasing) sequence of the computed order. Due to the recursive nature of the functionals, we can go further than that and obtain the same property for every $e' > e$. Indeed, for every $e' > e$, all the elements ordered by $\Gamma^{A}_{inc}(e)$, except for a finite amount, have been ordered by one of the (finitely many) recursive call to some $\Gamma^{A_i}_{inc}(e')$ for some $A_i \subseteq A$. If $W_{e'} \cap B$ was an infinite increasing sequence, then this would yield an infinite decreasing subsequence for one of the $\Gamma^{A_i}_{inc}(e')$, contradicting what was proved. \\

Thus, taking $\Gamma_{dec}^{\NN}(0)$ to be the ordering and $\Delta^{(\cdot) \oplus \NN}_{dec}(0)$ to be the functional yields the desired result.
\end{proof}

\begin{corollary}
There exists a computable linear ordering $(\NN, \sqsubseteq')$ with no computable infinite monotone sequence,
such that the measure of oracles computing both an infinite increasing and an infinite decreasing sequence is one.
\end{corollary}
\begin{proof}
Let $(\NN, \sqsubseteq)$ be the ordering obtained in \Cref{thm:linear-probabilistic-noncomputable} and let $(\NN, \sqsubseteq')$ be defined as follows: $a \sqsubseteq' b$ holds if either $a = 2a'$, $b = 2b'$ and $a' \sqsubseteq b'$, or $a = 2a'$ and $b = 2b'+1$, or $a = 2a'+1$, $b = 2b' + 1$ and $b' \sqsubseteq a'$.

The ordering $(\NN, \sqsubseteq')$ corresponds to two copies of $(\NN, \sqsubseteq)$ put one next to each other, with the second copy flipped. It is still a computable ordering with no infinite computable monotone sequence, as such a sequence would eventually stays in one of the copy an yield an infinite computable $\sqsubseteq$-monotone sequence.
\end{proof}

\begin{remark}
The result of \Cref{thm:linear-probabilistic-noncomputable} is not uniform in the sense that we obtain for every $\epsilon > 0$ a functional that computes an infinite increasing sequence for a set of oracles of measure $> 1 - \epsilon$, but codes for such functionals cannot be uniformly found in $\epsilon$. 

This can be solved by noticing that the proof of \Cref{thm:linear-probabilistic-noncomputable} allows us to uniformly compute an infinite sequence $(\NN, \sqsubseteq_n)_{n \in \NN}$ of linear orders with no computable infinite monotone sequence, as well as an infinite list of functionals $\Phi_{e(n)}$ computing an infinite increasing sequence for a set of oracles of measure $> 1 - \frac{1}{n+1}$. To do so, notice that the probability $\Pi_{\ell > 0} (1 - \frac{1}{2^{\ell}}) > \frac{1}{4}$ obtained in the proof could have been made arbitrarily close to $1$ by splitting into $2^{e+n+1}$ blocks instead of $2^{e+1}$ at stage $e$ of the construction.

We can then combine the orderings $(\NN, \sqsubseteq_n)_{n \in \NN}$ into a single ordering $(\NN, \sqsubseteq)$ as follows: let $\langle a,b \rangle \sqsubseteq \langle a',b' \rangle$ if either $b \sqsubset_0 b'$ or if $b = b'$ and $a \sqsubseteq_{b} a'$ (the order $(\NN, \sqsubseteq)$ correspond to the union of the all the orders $(\NN, \sqsubseteq_n)_{n \in \NN}$ placed relative to each other according to $(\NN, \sqsubseteq_0)$). 

The order $(\NN, \sqsubseteq)$ is computable and contains no infinite computable monotone sequence $(\langle a_n,b_n\rangle)_{n \in \NN}$. Indeed, if such a sequence existed, either the value of $b_n$ would change infinitely many time, yielding an infinite computable $\sqsubseteq_0$-monotone sequence, or eventually $b_n$ stays constant to some value $b$ after some stage $k$, and $(a_n)_{n > k}$ would be an infinite computable $\sqsubseteq_b$-monotone sequence. 

Finally, we obtain a single functional $\Phi^A(n)$, taking a parameter $n$ and running $\Phi^A_{e(n)}$ on the order induced on $\{\langle a,n \rangle : a \in \NN\}$. Thus, for every $n \in \NN$, the set of oracles $A$ such that $\Phi^A(n)$ is an infinite increasing sequence if of measure at least $1 - \frac{1}{n+1}$.
\end{remark}

\section{Separable permutations and first-order part}\label[section]{sec:sep-perm-first-order}

The \emph{first-order part} of a second-order system is the set of its first-order consequences, that is, the set of theorems in the language of first-order arithmetic. Characterizing the first-order part of a theorem is of fundamental importance in understanding its reverse mathematical strength. To this end, one usually proves that a second-order theorem~$T_0$ is a $\Pi^1_1$-conservative extension of another theorem~$T_1$ for which the first-order part is already known. If so, then their first-order parts coincide.

The characterization of the first-order part of Ramsey's theorem for pairs and two colors is one of the major remaining open questions around the reverse mathematics of combinatorial theorems. See Ko\l odziejczyk and Yokoyama~\cite{kolo2021search} for a survey on the subject. The first-order part of~$\RT^2_2$ is known to follow strictly from $\Sigma_2$-induction ($\mathsf{I}\Sigma_2$) and to imply the $\Sigma_2$-bounding scheme ($\mathsf{B}\Sigma_2$). It is currently unknown whether it coincides with $\mathsf{B}\Sigma_2$. The goal of this section is to relate the first-order parts of $\RT^2_2$ and $\RT^2_2(p)$ for any separable permutation~$p$.

By \Cref{prop:separable-implies-ads}, $\RCA_0 \vdash \RT^2_2(p) \to \ADS$ for every separable permutation~$p$.
It follows from Hirschfeldt and Shore~\cite{hirschfeldt2007combinatorial} that $\RT^2_2(p)$ implies the $\Sigma^0_2$-bounding scheme over~$\RCA_0$. We now generalize this to every permutation via a direct proof.

\begin{proposition}\label[proposition]{prop:rt22p-implies-bsig2}
    Let $p$ be a permutation. Then $\RCA_0 \vdash \RT_2^2(p) \to \BSig_2$.
\end{proposition}

\begin{proof}
$\BSig_2$ being equivalent to $\RT^1$ modulo $\RCA_0$, consider by contrapositive some coloring $f : \NN \to a$ for some $a \in \NN$ with no infinite $f$-homogeneous subset.

Let $g : [\NN]^2 \to 2$ defined by $g(x,y) = 0$ if $f(x) < f(y)$ and $g(x,y) = 1$ otherwise. We claim that for every permutation $p$, there is no infinite set $g$-avoiding $p$. \\

Firstly, by external induction on $n$, we prove that for every infinite set $H = \{x_0, x_1, \dots \}$ there exists some $x_{i_0} < \dots < x_{i_{n-1}}$ such that $f(x_{i_0}) < \dots < f(x_{i_{n-1}})$. Note that if there exists one such $n$-tuple, there exists an infinite amount of them, by truncating the beginning of $H$. 

If $n = 1$ (or $n = 0$), the property holds. So assume the property to be true for some $n \in \omega$ and fix such an infinite set $H = \{x_0, x_1, \dots \}$. Let $H' \subseteq H$ be the set of all $x_{i_{n-1}} \in H$ such that there exists some $x_{i_0} < \dots < x_{i_{n-1}} \in H$ with $f(x_{i_0}) < \dots < f(x_{i_{n-1}})$. By the inductive hypothesis, this set is infinite. If for some $x_{i_{n-1}} < x_{j_{n-1}} \in H'$ we have $f(x_{i_{n-1}}) < f(x_{j_{n-1}})$, then $x_{i_0} < x_{i_1} < \dots < x_{i_{n-1}} < x_{j_{n-1}}$ is the desired sequence. Otherwise, the values taken by $f$ on the elements of $H'$ form an infinite non-increasing sequence of integer smaller than $a$, which must eventually be constant by $\ISig_1$, contradicting the assumption that no infinite $f$-homogeneous set exists.
This proves that for every standard integer $n$, no infinite set can $g$ avoid the homogeneous set of size $n$ and of color $1$. \\

Let $p : n \to n$ be a permutation on $n$ elements and let $H \subseteq \NN$ be an infinite set.
There exists an infinite sequence of $n$-tuples of elements of $H$
$$(x_0^0 < \dots < x_{n-1}^0) < (x_0^1 < \dots < x_{n-1}^1) < \dots$$
each $g$-homogeneous for the color $1$. Notice that for a given n-tuple $(x_0, \dots, x_{n-1})$, there are $a^n$ possible values for $(f(x_0), \dots, f(x_{n-1}))$, so, by the finite pigeonhole principle for $a^n$-colors (provable in $\RCA_0$), there exists a sequence $k_0 < \dots < k_{n-1}$ such that $f(x_i^{k_{\ell}}) = f(x_i^{k_{\ell'}})$ for every $\ell < \ell' < n$ and every $i < n$. Then, the set $\{x_0^{k_{p(0)}} < x_1^{k_{p(1)}} < \dots < x_{n-1}^{k_{p(n-1)}}\}$ $g$-realizes $p$.
\end{proof}

In the quest for characterizing the first-order part of Ramsey's theorem for pairs, Patey and Yokoyama~\cite{patey2018proof} defined a combinatorial principle, namely, the grouping principle, as follows:
A \emph{largeness notion} is a collection $\LL$ of finite sets which is closed under superset, and such that every infinite set contains a finite subset in~$\LL$. For instance, the collection $\LL_\omega = \{ F : \card F > \min F \}$ is a largeness notion. Any element in~$\LL$ is called \emph{$\LL$-large}.

\begin{definition}[$\LL$-grouping]
Let $\LL$ be a largeness notion and $f : [\mathbb{N}]^2 \to 2$ a coloring.
A (finite or infinite) sequence of $\LL$-large sets $F_0 < F_1 < \dots$ of length $k \in \NN \cup \{\infty\}$ is an \emph{$\LL$-grouping} for~$f$ if for every~$i < j < k$, and every $x_0, x_1 \in F_i$ and $y_0, y_1 \in F_j$, $f(x_0, y_0) = f(x_1, y_1)$.
\end{definition}


\begin{statement}[Grouping principle]
$\GP^n_k$ is the statement \qt{For every largeness notion~$\LL$ and every coloring $f : [\NN]^n \to k$, there is an infinite $\LL$-grouping.}
\end{statement}

Kreuzer (see \cite[Theorem 9.1]{patey2018proof}) proved that $\RCA_0 + \GP^2_2$ implies $\BSig_2$ over~$\RCA_0$.
Le Houérou, Patey and Yokoyama~\cite{houerou2023conservation} later proved that $\RCA_0 + \GP^2_2$ is a $\Pi^1_1$-conservative extension of~$\RCA_0 + \BSig_2$.

\begin{proposition}\label[proposition]{prop:gp22-cort22}
For every separable permutation~$p$, $\RCA_0 \vdash \GP^2_2 \to \coRT^2_2(p)$.
\end{proposition}
\begin{proof}
By external induction on the length of~$p$.

This property is vacuously true for the pattern of length $1$, so assume that $|p| > 1$ and that the property holds for all pattern of strictly smaller length.
Let $f : [\NN]^2 \to 2$ be an instance of $\coRT^2_2(p)$. There are two cases:
\smallskip

\textbf{Case 1:} $p$ is reducible. Hence, there exists two separable patterns $p_0,p_1$ such that $p_0 \uplus p_1 = p$. Since $f$ avoids $p$, then by Mimouni and Patey~\cite{mimouni2025ramseylike}, there is an infinite set $G = \{ x_0 < x_1 < \dots \}$ $f$-avoiding $p_i$ for some $i < 2$. Let $g : [\NN]^2 \to 2$ be defined by $g(n, m) = f(x_n, x_m)$. In particular, $g$ avoids~$p_i$, so by the induction hypothesis, there is an infinite $g$-homogeneous subset~$H$. The set $\{ x_n : n \in H \}$ is an infinite $f$-homogeneous set.
\smallskip

\textbf{Case 2:} $p$ is convergent, say $p = p^- \rhd c$ for some~$c < 2$. If there is some infinite subset $H \subseteq \NN$ avoiding the pattern $p^-$, then we can apply the induction hypothesis for $p^-$ on the coloring induced by $f$ on the set $H$, which would give us an infinite $f$-homogeneous set.
So assume that no such $H$ exists, hence 
$$L_{p^-} := \{S \subseteq \NN : \exists S' \subseteq S,\ S'~ f\textit{-realizes } p^-\}$$ forms a largeness notion. By $\GP_2^2$, there exists an infinite $L_{p^-}$-grouping $F_0 < F_1 < \dots$ for $f$. By definition of a grouping, for every~$i < j$, every~$x \in F_i$ and $y \in F_j$, $f(x, y) = f(\min F_i, \min F_j)$. Since $f$ avoids~$p$ but $F_i$ $f$-realizes $p^-$, then $f(\min F_i, \min F_j) = 1-c$. Thus, the set $\{ \min F_i : i \in \NN \}$ is $f$-homogeneous for color~$1 - c$. 
\end{proof}

\begin{corollary}\label[corollary]{cor:rt22p-gp-rt22}
    For every separable permutation $p$, $\RCA_0 \vdash \RT^2_2 \leftrightarrow \GP_2^2 + \RT_2^2(p)$.
    Moreover, $\RT^2_2$ and $\SepRT^2_2 + \GP^2_2$ are equivalent over $\omega$-models.
\end{corollary}

\begin{proof}
The implications $\RCA_0 \vdash \RT^2_2 \rightarrow \GP_2^2 + \RT_2^2(p)$ and $\RCA_0 \vdash \RT^2_2 \rightarrow \GP_2^2 + \SepRT_2^2$ are clear.
Conversely, let $p$ be a separable permutation. Since $\RCA_0 \vdash \RT^2_2(p) + \coRT^2_2(p) \rightarrow \RT^2_2$, then by \Cref{prop:gp22-cort22}, $\RCA_0 \vdash \RT^2_2(p) + \GP^2_2 \rightarrow \RT^2_2$.
Last, let $\M$ be an $\omega$-model of~$\SepRT^2_2 + \GP^2_2$. By \Cref{prop:gp22-cort22}, $\M \models \coRT^2_2(p)$ for every separable permutation~$p$. Since $\RCA_0 \vdash (\SepRT^2_2 \wedge \forall p \mbox{ separable permutation } \coRT^2_2(p)) \to \RT^2_2$, then $\M \models \RT^2_2$.

\end{proof}

\begin{corollary}
For every separable permutation~$p$, $\RCA_0 + \RT^2_2$ is $\Pi^1_1$-conservative over~$\RCA_0 + \BSig_2$ iff $\RCA_0 + \RT^2_2(p)$ is $\Pi^1_1$-conservative over~$\RCA_0 + \BSig_2$.
\end{corollary}
\begin{proof}
If $\RCA_0 + \RT^2_2$ is $\Pi^1_1$-conservative over~$\RCA_0 + \BSig_2$, then since  $\RCA_0 \vdash \RT^2_2 \to \RT^2_2(p)$, so is $\RCA_0 + \RT^2_2(p)$. Suppose now that $\RCA_0 + \RT^2_2(p)$ is $\Pi^1_1$-conservative over~$\RCA_0 + \BSig_2$. Yokoyama~\cite{yokoyama2010conservativity} proved the following amalgamation theorem: if $T_0, T_1, T_2$ are $\Pi^1_2$ theories such that $T_0 \supseteq \RCA_0$ and $T_1$ and $T_2$ are $\Pi^1_1$-conservative extensions of~$T_0$, then $T_1 + T_2$ is a $\Pi^1_1$-conservative extension of~$T_0$. Letting $T_0 = \RCA_0 + \BSig_2$, $T_1 = \RT^2_2(p)$ and $T_2 = \GP^2_2$, then by Le Houérou, Levy Patey and Yokoyama~\cite{houerou2023conservation} and \Cref{cor:rt22p-gp-rt22}, $\RCA_0 + \RT^2_2$ is $\Pi^1_1$-conservative over~$\RCA_0 + \BSig_2$.
\end{proof}

By Mimouni and Patey~\cite{mimouni2025ramseylike}, $\RT^2_2(1302)$ admits preservation of so-called $\omega$ 2-dim hyperimmunities, and therefore does not imply $\EM$ over $\omega$-models. 
It is currently unknown whether $\EM$ implies $\RT^2_2(1302)$ over $\RCA_0$ and whether $\RT^2_2$ forms a strong minimal cover of~$\EM$. One possible direction consists in studying the strength of $\EM \wedge \RT^2_2(1302)$ so see whether it implies some statement known to be stronger than~$\EM$. We now prove that $\EM + \RT^2_2(1302)$ implies the grouping principle for pairs for $\omega^n$-largeness ($\GP^2_2(\LL_{\omega^n})$). 
The following definition is a parti\-cular case of $\alpha$-largeness, a quantitative notion of largeness defined by Ketonen and Solovay~\cite{ketonen1981rapidly} to better understand the unprovability of some combinatorial theorems.

\begin{definition}
A finite set~$F \subseteq \NN$ is
\begin{itemize}
    \item[(1)] \emph{$\omega^0$-large} if $F \neq \emptyset$;
    \item[(2)] \emph{$\omega^{n+1}$-large} if there are $\min F$ many $\omega^n$-large finite subsets of~$F \setminus \{\min F\}$
    $$
    E_0 < E_1 < \dots < E_{\min F-1}
    $$
\end{itemize}
\end{definition}

Given $n \in \NN$, let $\LL_{\omega^n}$ be the set of all $\omega^n$-large sets.
These notions of $\omega^n$-largeness for $n \in \NN$ are very useful to prove $\forall \Pi^0_3$-conservation statements over $\RCA_0$ (see \cite{patey2018proof,kolo2020some}). We shall in particular use the following theorem.

\begin{theorem}[Patey and Yokoyama~\cite{patey2018proof}]\label[theorem]{thm:rt22-omegan-largeness}
For every~$n \in \NN$, $\RCA_0$ proves \qt{For every coloring $f : [\NN]^2 \to 2$, there is an $\omega^n$-large $f$-homogeneous set}.
\end{theorem}

We are now ready to prove the claimed proposition.

\begin{proposition}\label[proposition]{prop:em-1302-gp22}
For every~$n \in \NN$, $\RCA_0 \vdash \EM + \RT^2_2(1302) \to \GP^2_2(\LL_{\omega^n})$.
\end{proposition}
\begin{proof}
Fix $n \in \omega$. First, note that $\EM$ (see Kreuzer~\cite{kreuzer2012primitive}) and $\RT^2_2(1302)$ (see \Cref{prop:rt22p-implies-bsig2}) both imply~$\BSig_2$.
Let $f : [\NN]^2 \to 2$ be a coloring.
Since $\RCA_0 \vdash \RT^2_2(1302) \leftrightarrow \RT^2_2(2031)$, then by $\EM$ and $\RT^2_2(1302)$, there exists an infinite $f$-transitive set~$X$ which $f$-avoids~2031 and~1302, that is, avoids the pattern (\ref{eq:pattern-2031}) and the pattern (\ref{eq:pattern-1302}) below:
\vspace{1.3cm}
\begin{equation}\label{eq:pattern-1302}
\patternfour{22}    
\end{equation}
By \Cref{thm:rt22-omegan-largeness}, there is an infinite set~$Y \subseteq X$ and some color~$i < 2$ 
such that every infinite set $Z \subseteq Y$ contains an $\omega^n$-large $f$-homogeneous subset for color~$i$. Indeed, suppose that the set~$X$ does not satisfy the desired property for $i = 0$. Then there is an infinite subset~$Y \subseteq X$ which does not contain any $\omega^n$-large $f$-homogeneous subset for color~0. Then, by  \Cref{thm:rt22-omegan-largeness}, $Y$ has the desired property for $i = 1$.
\smallskip

\textbf{Claim 1:} For every $a < b < c \in Y$ such that $f(a, b) = i$ and $f(a, c) = f(b, c) = 1-i$,
then for every~$d \in H$ with $d > c$, we have $f(a, d) = f(b, d)$.
Indeed, if $f(a, d) = 1-i$ but $f(b, d) = i$, the 3-cycle $\{a, b, d\}$ would contradict $f$-transitivity of~$Y$ for color~$i$. Moreover, if $f(a, d) = i$ and $f(b, d) = 1-i$, then since $Y$ $f$-avoids~1302 (for $i=0)$ and 2031 (for $i=1$), $f(c, d) = 1-i$, and the 3-cycle $\{a, c, d\}$ contradicts $f$-transitivity of $Y$ for color~$1-i$. This proves our claim.
\smallskip

We now have two cases.

\smallskip
\textbf{Case 1:} There is some infinite subset $Z \subseteq Y$ such that for every $\omega^n$-large $f$-homogeneous subset $F \subseteq Z$ for color~$i$, $\lim_{y \in Z} f(\min F, y) = i$.
In this case, let $F_0 < F_1 < \dots$ be an infinite sequence of $\omega^n$-large $f$-homogeneous subsets of~$Z$ for color~$i$. The set $\{ \min F_s : s \in \NN \}$ is an infinite set such that every element has limit~$i$ in~$Z$. Then, using $\BSig_2$, there is an infinite $f$-homogeneous subset for color~$i$, and \emph{a fortiori} an infinite $\omega^n$-grouping for~$f$.

\smallskip
\textbf{Case 2:} Case 1 does not hold. We then create an infinite sequence $F_0 < x_0 < F_1 < x_1 < \dots$ such that for every~$s \in \NN$,
\begin{itemize}
    \item[(a)] $F_s$ is an $\LL$-large $f$-homogeneous subset of~$Y$ for color~$i$;
    \item[(b)] $x_s \in Y$ is such that $f(\min F_s, x_s) = 1-i$;
    \item[(c)] for every~$t > s$, $F_t$ is homogeneous for the coloring $y \mapsto f(\min F_s, y)$.
\end{itemize}
Assume $F_0 < x_0 < \dots < F_k < x_k$ are defined. Search $H$-computably for some~$F_{k+1} < x_{k+1}$ such that the sequence satisfies (a-c). We claim that such $F_{k+1} < x_{k+1}$ exist. Indeed, by $\RT^1$ which follows from $\BSig_2$, there is an infinite set~$Z \subseteq Y$ which is simultaneously homogeneous for every coloring $y \mapsto f(\min F_\ell, y)$ for every~$\ell \leq k$. By choice of~$Y$ and since Case~1 does not hold, there is an $\omega^n$-large $f$-homogeneous set $F_{k+1} \subseteq Z$ for color~$i$, and some~$x_{k+1} \in Z$ with $x_{k+1} > F_{k+1}$ such that $f(\min F_{k+1}, x_{k+1}) = 1-i$. This completes the construction.

We claim that the sequence $(F_s)_{s \in \NN}$ is an $\omega^n$-grouping for~$f$. Fix $s < t \in \NN$. Since $F_s$ is $f$-homogeneous for color~$i$ and $f(\min F_s, x_s) = 1-i$, then by $f$-transitivity of~$Y$ for color~0, for every~$y \in F_s$, $f(y, x_s) = 1-i$. It follows from Claim~1 that for every~$y \in F_s$ and every~$z \in Y$ with $z > x_s$, $f(\min F_s, z) = f(y, z)$. Last, since $F_t$ is homogeneous for $z \mapsto f(\min F_s, z)$, then for every $y \in F_s$ and $z \in F_t$, $f(y, z) = f(\min F_s, \min F_t)$. This completes the proof of \Cref{prop:em-1302-gp22}.
\end{proof}

The remainder of this section is dedicated to the proof that $\EM$ implies $\GP^2_2$ over $\omega$-models.

\begin{lemma}\label[lemma]{lem:finite-increasing-large-sequence}
   Let $\LL$ be a computable largeness notion and $\L = (\NN, <_\L)$ a computable linear order, then, for every $k \in \NN$, there exists some sequence $F_0, \dots, F_{k-1}$ of $\LL$-large sets, such that $\max_\L(F_i) <_\L \min_\L(F_{i+1})$ for every $i < k-1$.
\end{lemma}

\begin{proof}
By induction on $k$, we prove the stronger property that for infinitely many $x \in \NN$, there exists some sequence $F_0, \dots, F_{k-1}$ of $\LL$-large sets, such that $\max_\L(F_i) <_\L \min_\L(F_{i+1})$ for every $i < k-1$ and such that $\max_{\L} F_{k-1} <_\L x$. \\

The property is vacuously true for $k = 0$.

Assume the property to be true for some $k \geq 0$. We can then computably find some infinite sequence $x_0 < x_1 < \dots$ such that for every $\ell \in \NN$, there is some sequence $F_0^\ell, \dots, F_{k - 1}^\ell$ of $\LL$-large sets, with $\max_\L(F_i^\ell) <_\L \min_\L(F_{i+1}^\ell)$ for every $i < k-1$ and such that $\max_{\L}(F_{k - 1}^\ell) <_{\L} x_\ell$. 

Define inductively the sets $(H_n)_{n \in \NN}$ as follows: let $H_0 = \{x_0, x_1, \dots \}$, then, for every $n \in \NN$, if $H_n$ has some $\L$-maximal element $x$, let $H_{n+1} = H_n \setminus \{x\}$, otherwise let $H_{n+1} = H_n$. Every $H_n$ is an infinite set, as we only remove at most one element each time.

By definition of a largeness notion, for every $n$, there exists some $\LL$-large subset $F \subseteq H_n$. Let $\ell \in \NN$ be such that $\min_{\L}(F) = x_\ell$, then $\max_{L}(F_{k - 1}^\ell) <_\L \min_{\L}(F)$ and there are at least $n$ elements $x \in \NN$ that are $\L$-bigger than $\max_{\L} F$ by definition of $H_n$ (if $H_n$ has no maximal element, then this is obvious, and is $H_n$ has some maximal element, then there are $n$ elements bigger than it in $H_0$). Hence the property is true for $k+1$.
\end{proof}

Note that by \Cref{thm:rt22-omegan-largeness}, the previous lemma can be proven over $\RCA_0$ for $\omega^n$-largeness. Indeed, any $\omega^{n+1}$-large $\L$-homogeneous set~$F$ can be decomposed into $\min F$ many $\omega^n$-large sets with the desired property.

The following proposition uses the same proof structure as in \Cref{sec:2dnc-sep}.

\begin{proposition}\label[proposition]{prop:2-DNC-linear-grouping}
    Let $\LL$ be a computable largeness notion and $\L = (\NN, <_\L)$ a computable linear order. Then, every $\emptyset'$-DNC degree $\dbf$ computes an infinite $\LL$-grouping.
\end{proposition}

\begin{proof}
If $\dbf \geq_T \emptyset'$, then we are done, as every computable linear order admits a $\Pi^0_1$ infinite $\L$-monotonous sequence (Manaster; see Downey~\cite[Section 5]{downey1998computability}). The degree $\dbf$ can therefore compute such a sequence, and from it, any infinite sequence of $\LL$-large subsets forms an $\LL$-grouping.

If $\dbf \not \geq_T \emptyset'$, then, let $g \leq_T \dbf$ be the function of \Cref{lem:equiv-dnc-avoid}(2). Let $\mu_s$ be the standard left-c.e.\ approximation of the modulus $\mu$ of $\emptyset'$, that is, $\mu_s(x)$ is the least~$t \leq s$ such that $\emptyset'_t \uh x = \emptyset'_s \uh x$. For every $x \in \NN$, the sequence $(\mu_s(x))_{s \in \NN}$ can only change values at most $x$ times (every time a program with Turing index less than $x$ stops), and let $m^x_0 < \dots < m^x_{i_x} (= \mu(x))$ be those values for some $i_x < x$.

Every set $F \subseteq \NN$ defines an interval comprised of all the elements $x$ such that $\min_\L F <_\L x < \max_\L F$. For $H_0, \dots, H_{s - 1}$ a family of finite sets, let $X_{\langle H_0, \dots, H_{s - 1} \rangle}$ be set of all the elements bigger (for the standard order on $\NN$) than $\max H_{s-1}$ and not inside one of the interval defined by the $H_i$ for $i < s$.

A finite set $F$ will be said to be \emph{good} in $X_{\langle H_0, \dots, H_{s - 1} \rangle}$ if there are infinitely many $x \in X_{\langle H_0, \dots, H_{s - 1} \rangle}$ not in the interval determined by $F$, otherwise $F$ is said to be \emph{bad} in $X_{\langle H_0, \dots, H_{s - 1} \rangle}$. Notice that being bad is a $\Sigma_2^0$ property, and that if $F$ is good in $X_{\langle H_0, \dots, H_{s - 1} \rangle}$, then $X_{\langle H_0, \dots, H_{s - 1}, F \rangle}$ is infinite.

Given some infinite set $X_{\langle H_0, \dots, H_{s - 1} \rangle}$ and some sequence $F_0, \dots, F_{k-1}$ of finite sets such that $\max_\L(F_i) <_\L \min_\L(F_{i+1})$ for every $i < k-1$, at most one $F_i$ can be bad in $X_{\langle H_0, \dots, H_{s - 1} \rangle}$. Indeed, under our assumption, the intervals defined by the $F_i$ are disjoints, hence no two of them can have a finite complement in $X_{\langle H_0, \dots, H_{s - 1} \rangle}$. \\

We can $\dbf$-computably build an infinite sequence of finite sets $H_0 < H_1 < \dots$ satisfying that:
\begin{itemize}
    \item $X_{\langle H_0, \dots, H_{s - 1} \rangle}$ is infinite for every $s \in \NN$.
    \item $H_0, \dots, H_{s-1}$ forms an $\LL$-grouping for every $s \in \NN$.
\end{itemize}
 
Suppose $H_0, \dots, H_{s-1}$ have been defined for some $s \in \NN$. Since they form an $\LL$-grouping, their corresponding intervals are disjoints and $X_{\langle H_0, \dots, H_{s - 1} \rangle}$ can be partitioned into $s+1$ many blocks $X^0, \dots, X^{s}$ with $X^i$ containing all the $x \in X_{\langle H_0, \dots, H_{s - 1} \rangle}$ such that $\max_{\L} H_j <_\L x$ for exactly $i$ values of $j$.

For every $x \in \NN$, let $W_{e(x,\langle H_0, \dots, H_{s-1} \rangle)}^{\emptyset'}$ be the $\emptyset'$-c.e.\ set which, for every $j < i_x$, searches for some $\ell \leq s$ and some sequence $F_0, \dots, F_{m_j^x} \subseteq X^{\ell}$ of $\LL$-large sets, such that $\max_{\L}(F_i) <_{\L} \min_{\L}(F_{i+1})$ for every $i < m_j^x$ and then list the index $i \leq m_j^x$ such that $F_i$ is bad in $X_{\langle H_0, \dots, H_{s - 1} \rangle}$ if there is one.

Note that $\card W_{e(x,\langle H_0, \dots, H_{s-1} \rangle)} \leq x$ for every $x \in \NN$ and let $h(x) = g(e(x,\langle H_0, \dots, H_{s-1} \rangle), x)$. Since $\dbf \not \geq_T \emptyset'$ and $h$ is $\dbf$-computable, we have $h(x) < \mu(x)$ for infinitely many $x$, hence we can wait and find some $x$ and $j$ such that $h(x) < m_j^x$. 

Computably search for some $\ell \leq s$ and some sequence $F_0, \dots, F_{m_j^x} \subseteq X^{\ell}$ of $\LL$-large sets such that $\max_{\L}(F_i) <_{\L} \min_{\L}(F_{i+1})$ for every $i < m_j^x$. Such a sequence exists by \Cref{lem:finite-increasing-large-sequence} as one of the $X^{\ell}$ is infinite. Then, by definition of $g$, $F_{h(x)}$ will be good in $X_{\langle H_0, \dots, H_{s - 1} \rangle}$ . Pick $H_s = F_{h(x)}$. By our assumption that $F_{h(x)} \subseteq X^{\ell}$ for some $\ell$, we have that $H_0, \dots, H_s$ forms an $\LL$-grouping.
\end{proof}

\begin{theorem}
    $\EM$ implies $\GP^2_2$ in $\omega$-models.
\end{theorem}

\begin{proof}
Let $\M = (\NN,S)$ be an $\omega$-model of $\RCA_0 + \EM$.

Let $\LL$ be an $X$-computable largeness notion for some set $X \in S$ and let $f : [\NN]^2 \to 2$ be a $Y$-computable coloring for some set $Y \in S$. Using $\EM$, we can assume $f : [\NN]^2 \to 2$ to be a linear order. 

$\EM$ implies the existence of a set whose degree is $(X \oplus Y)'$-DNC, hence, by a relativized form of \Cref{prop:2-DNC-linear-grouping}, this set computes an infinite $\LL$-grouping for $f$, which will therefore be in $S$. Hence $\M \models \GP_2^2$.
\end{proof}

\section{Open questions}\label[section]{sec:open-questions}

Thanks to \Cref{maintheorem1}, in order to decide whether a Ramsey-like theorem implies $\RT^2_2$ over $\omega$-models, is now sufficient to check whether the corresponding forbidden pattern is a separable permutation, which is a finitistic verification. The characterization of the Ramsey-like theorems equivalent to Ramsey's theorem over $\omega$-models is only a first step towards the understanding of the structure of the Ramsey-like theorems partially ordered by the implication relation in reverse mathematics. The systematic study of Ramsey-like statements is relatively new, so many questions remain open. The known relations over $\omega$-models are given in \Cref{fig:diagram}.

\begin{figure}[htbp]
\begin{center}
\begin{tikzpicture}
    \node(A) at (3,0){$\RCA_0$};
    \node(B) at (3,6){$\RT_2^2$};
    \node(C) at (3,4){$\EM$};
    \node(D) at (-2,4){$\ADS$};
    \node(E) at (6,3){$\RT_2^2(p)$};
    \node(F) at (0,6){$\RT_2^2(p)$};
    \node(G) at (6,4){$\RT_2^2(1302)$};
    \node(H) at (1,3){$\RT_2^2(p)$};
    \node(I) at (1,4){$\RT_2^2(~~~~~~~)$};
    
    \node[fill=black, shape=circle,scale=0.35] (1) at (1,4) {};
	\node[fill=black, shape=circle,scale=0.35] (2) at (1.35,4) {};
	\node[fill=black, shape=circle,scale=0.35] (3) at (1.7,4) {}; 
    \path (1) edge[bend left,scale=0.2] node[below,yshift=2pt] {\tiny 0} (2);
    \path (2) edge[bend left,scale=0.2] node[below,yshift=2pt] {\tiny 0} (3);
    \path (1) edge[bend left=45,scale=0.2] node[above,yshift=-2.5pt] {\tiny 1} (3);
    \path (1) edge[bend right=45,opacity=0,scale=0.2] node {} (3);

    \node[rectangle,dashed,draw,minimum width = 2cm,minimum height = 2cm, opacity=0.75, label={[opacity=0.75, align=center]:\tiny\shortstack{\linespread{0.7}\selectfont non-separable\\ \tiny permutations}}] (r1) at (6,3.5) {};
    \node[rectangle,dashed,draw,minimum width = 2cm,minimum height =1cm, opacity=0.75, label={[opacity=0.75]:\tiny separable permutations}] (r2) at (0,6) {};
    \node[rectangle,dashed,draw,minimum width = 2.2cm,minimum height = 2cm, opacity=0.75, label={[opacity=0.75]:\tiny non permutations}] (r3) at (1,3.5) {};

    \draw [-stealth,double](B) -- (C);
    \draw [-stealth,double](B) -- (D);
    \draw [<->](C) -- (I);
    \draw [-stealth,double](H) -- (A);
    \draw [-stealth,double](D) -- (A);
    \draw [-stealth,double](E) -- (A);
    \draw [-stealth, negated] (G) -- (C) node[midway,above] {\tiny \cite{mimouni2025ramseylike}};
    \draw [-stealth, dashed, maybe] (C) to[out=-30, in=-160] (G);
    \draw [<->] (B) -- (F) node[midway,above] {\tiny Thm.~\ref{maintheorem1}};
    \draw [-stealth](G) -- (E);
    \draw [-stealth,double](B) -- (G) node[midway,above] {\tiny \cite{mimouni2025ramseylike}};
    \draw [-stealth](I) -- (H) node[midway,right] {\tiny \cite{mimouni2025ramseylike}};
\end{tikzpicture}\label[figure]{fig:diagram}
\caption{Implications over $\omega$-models of $\RCA_0$. A simple (double) arrow represents a (strict) implication over $\RCA_0$. A dashed arrow with a question mark represents denotes that the implication is still open and a crossed arrow represents a non-implication. \\ Note that by \cite{mimouni2025ramseylike}, there exist non-permutations $p$ such that $\RT^2_2(p)$ is strictly weaker than $\EM$ over $\omega$-models.}
\end{center}
\end{figure}

The most fundamental question is about the general structure of this partial order.

\begin{question}
What is the structure of the Ramsey-like statements, ordered by implication over~$\RCA_0$? over $\omega$-models? over computable reduction?
\end{question}

We showed in particular in \Cref{cor:em-minimal-cover} that $\RT^2_2$ forms a minimal cover of~$\EM$ over $\RCA_0$ in the set of Ramsey-like statements, and \emph{a fortiori} over $\omega$-models. It remains unknown whether this is actually a strong minimal cover. 

\begin{question}
Is $\RT^2_2$ a strong minimal cover of~$\EM$ over~$\RCA_0$? over $\omega$-models? over computable reduction?
\end{question}

In particular, over $\omega$-models, if $\RT^2_2(p)$ is strictly weaker than~$\RT^2_2$, then by \Cref{maintheorem1}, $p$ is not a separable permutation, so either it is not a permutation, in which case $\RT^2_2(p)$ follows from~$\EM$, or it is a permutation, but not separable, and by \Cref{thm:characterization-sep-bose} it follows from $\RT^2_2(1302)$. By Mimouni and Patey~\cite{mimouni2025ramseylike}, $\RT^2_2(1302)$ does not imply $\EM$ over~$\omega$-models, but the other direction is unknown.

\begin{question}
Does $\EM$ imply $\RT^2_2(1302)$ over $\omega$-models? 
\end{question}



The main theorem of this article states that $\RT^2_2(p)$ implies $\RT^2_2$ over $\omega$-models whenever~$p$ is a separable permutation. The proof however seems to make a strong use of $\omega$-models, in particular with \Cref{lem:unbalanced-ramsey}. We shall prove in a later paper that \Cref{lem:unbalanced-ramsey} is not provable over~$\BSig_2$.

\begin{question}
For every separable permutation~$p$, does $\RCA_0 \vdash \RT^2_2(p) \to \RT^2_2$?
\end{question}

Moreover, the proof uses two applications of $\RT^2_2(p)$. In particular, 120 and 102 are the two smallest separable permutations for which the following question is open.

\begin{question}
For every separable permutation~$p$, is $\RT^2_2$ computably reducible to $\RT^2_2(p)$?
\end{question}

To prove our main theorem, we show that for every computable instance of $\coRT^2_2(p)$, every 2-random computes a solution. However, the proof seems to make multiple essential uses of $\Sigma^0_2$-induction.

\begin{question}
For every separable permutation~$p$, does $\RCA_0 + \BSig_2 \vdash \RANs{2} \to \coRT^2_2(p)$?
\end{question}

The statement $\RANs{2}$ was studied by various authors in reverse mathematics~\cite{conidis2013random,nies2020randomness,belanger2021where}. In particular, Conidis and Slaman~\cite{conidis2013random} showed that $\BSig_2 + \RANs{2}$ is equivalent to $\WWKLs{2}$ over~$\RCA_0$, and is $\Pi^1_1$-conservative over~$\RCA_0 + \BSig_2$, where $\WWKLs{2}$ is weak K\"onig's lemma for $\Delta^0_2$ trees of positive measure.

\bigskip
\begin{center}
\textbf{Acknowledgements}
\end{center}
The authors are thankful to the anonymous referees for insightful comments and improvement suggestions.

\bibliographystyle{plain}
\bibliography{biblio}

\end{document}